\documentclass[12pt]{amsart}

\usepackage{amsmath,amsthm,amsfonts,amssymb,amscd}
\usepackage[all]{xy}
\usepackage{sseq}

\author{Thomas Baird}
\title{Moduli spaces of vector bundles over a real curve: $\Z/2$-Betti numbers}

\theoremstyle{plain}
\newtheorem{thm}{Theorem}[section]
\newtheorem{prop}[thm]{Proposition}
\newtheorem{lem}[thm]{Lemma}
\newtheorem{cor}[thm]{Corollary}

\theoremstyle{definition}
\newtheorem{definition}[equation]{Definition}

\newtheorem{remark}[equation]{Remark}

\theoremstyle{remark}

\numberwithin{equation}{section}

\newcommand{\gau}{\mathcal{G}}

\newcommand{\D}{\mathcal{D}}

\newcommand{\Z}{\mathbb{Z}}

\newcommand{\R}{\mathbb{R}}

\newcommand{\C}{\mathbb{C}}

\newcommand{\rk}{\mathrm{rank}}

\renewcommand{\phi}{\varphi}

\newcommand{\im}{\mathrm{im}}

\newcommand{\M}{\mathcal{M}}

\DeclareMathSymbol{\sdp}{\mathbin}{AMSb}{"6F}

\newcommand{\G}{\mathcal{G}}

\newcommand{\E}{\mathcal{E}}
\newcommand{\F}{\mathcal{F}}
\newcommand\ca{\mathcal}
\newcommand{\ignore}[1]{}

\newcommand{\rank}{\mathrm{rank}}
\begin{document}
\baselineskip=16pt

\maketitle

\begin{abstract}
Moduli spaces of real bundles over a real curve arise naturally as Lagrangian submanifolds of the moduli space of semi-stable bundles over a complex curve. In this paper, we adapt the methods of Atiyah-Bott's ``Yang-Mills over a Riemann Surface'' to compute $\Z/2$-Betti numbers of these spaces.  MR 32L05, 14P25.
\end{abstract}

\section{Introduction}\label{intro}
\subsection{Background}
A \emph{real curve} $(\Sigma, \sigma)$ is a closed, complex $1$-manifold $\Sigma = (\Sigma, J)$ equipped with a $C^{\infty}$-map $$\sigma: \Sigma \rightarrow \Sigma$$ such that $\sigma^2 = Id_{\Sigma}$ and $d\sigma \circ J = - J \circ d\sigma$ (we suppress $J$ in our notation throughout). The map $\sigma$ is called the \emph{anti-holomorphic involution} and the fixed point set $\Sigma^{\sigma}$ is called the set of \emph{real points} of $(\Sigma, \sigma)$. 

Given relatively prime integers $r$ and $d$ with $r\geq 1$,  there exists a non-singular projective moduli space $M_{\Sigma}(r,d)$ classifying stable holomorphic bundles of rank $r$ and degree $d$ over the underlying complex curve $\Sigma$ (Mumford \cite{mumford2004projective}).  The anti-holomorphic involution $\sigma$ induces an anti-holomorphic involution on $M_{\Sigma}(r,d)$ sending (the isomorphism class of) the holomorphic bundle $\E \rightarrow \Sigma$ to the bundle $$\sigma(\E) = \overline{\sigma^*\E}.$$ The set of fixed points $M_{\Sigma}(r,d)^{\sigma}$ is a real submanifold that is Lagrangian with respect to a natural Kaehler structure on $M_{\Sigma}(r,d)$. The main result of this paper is a recursive formula for the $\Z_2$-Betti numbers of the path components of $M_{\Sigma}(r,d)^{\sigma}$.

The case of rank $r=1$ was considered by Gross-Harris \cite{gross1981real}. Recall that $$M_{\Sigma}(1,d) = Pic_d(\Sigma)$$ is homeomorphic to a compact torus $(S^1)^{2g}$, where $g$ is the genus of $\Sigma$. For a divisor class $[D] \in Pic_d(\Sigma)$, the involution satisfies $\sigma([D]) = [\sigma(D)].$ The fixed point set $Pic(\Sigma)^{\sigma}$ is a disjoint union of Lagrangian tori each diffeomorphic to $(S^1)^g$.

The general rank case was studied in independent papers by Biswas-Huisman-Hurtubise \cite{biswas2009moduli} and Schaffhauser \cite{schaffhauser2009moduli}. They proved that the fixed points lying in $M_{\Sigma}(r,d)^{\sigma}$ correspond to bundles admitting an antiholomorphic lift

$$\xymatrix{ \E \ar[r]^{\tau} \ar[d] & \E \ar[d] \\
            \Sigma \ar[r]^{\sigma} & \Sigma } $$ 
such that either 
\begin{itemize}
	\item[(a)] $\tau^2 = Id_\E$,~~~~~~in which case we call $(\E,\tau)$ a \emph{real bundle} over $(\Sigma,\sigma)$, or 
	\item[(b)] $\tau^2 = -Id_\E$,~~~~in which case we call $(\E,\tau)$ a \emph{quaterionic bundle} over $(\Sigma,\sigma)$.
\end{itemize}

The axioms defining real and quaterionic bundles make sense for $C^{\infty}$-bundles $E\rightarrow \Sigma$ as well as for holomorphic ones. The authors \cite{biswas2009moduli} and \cite{schaffhauser2009moduli} proved that the path components of  $M_{\Sigma}(r,d)^{\sigma}$ are classified by isomorphism types of real and quaterionic $C^{\infty}$-bundles.

Given a real curve $(\Sigma, \sigma)$, the set of real points $\Sigma^{\sigma}$ is a finite union of circles. If $(E,\tau)\rightarrow (\Sigma, \sigma)$ is a real $C^{\infty}$-bundle, then the fixed point set $E^{\tau}$ forms a $\R^r$-bundle over $\Sigma^{\sigma}$.  We paraphrase Prop 4.1 and Prop 4.2 of \cite{biswas2009moduli}.

\begin{thm}\label{arecbu}
Real $C^{\infty}$-vector bundles $(E,\tau)$ over a real curve $(\Sigma, \sigma)$ are classified up to isomorphism by rank $r$, degree $d$ and Stieffel-Whitney class $w_1(E^{\tau}) \in H^1(\Sigma^{\sigma};\Z_2)$ subject to the condition that $$d \equiv w_1(E^{\tau})(\Sigma^{\sigma}) ~mod~2. $$ 
Quaternionic vector bundles are classified by rank $r$ and degree $d$, subject to the condition
\begin{equation}\label{quatcond}
d \equiv r(g-1) ~mod~2
\end{equation}
and that $\Sigma^{\sigma} = \emptyset$ if $r$ is odd.
\end{thm}

\begin{remark}\label{needremakr}
Condition (\ref{quatcond}) implies that a real curve $(\Sigma, \sigma)$ admits a quaternionic vector bundle of rank coprime rank and degree if and only if it admits a quaternionic line bundle.
\end{remark}

The strategy of the current paper (pursued independently by Liu-Schaffhauser \cite{liu2011yang}) is to adapt the methods of Atiyah-Bott \cite{ab2} to compute the $\Z/2$-Betti numbers of path components of $M_{\Sigma}(r,d)^{\sigma}$. We outline this approach in the following section.

\subsection{The Atiyah-Bott argument}

The \emph{slope} of a holomorphic vector bundle $\E\rightarrow \Sigma$ is the ratio of the degree to the rank, $$\mu(\E) := \deg(\E)/\rk(\E) =\deg(E)/\rk(E)=d/r .$$
The bundle $\E$ is called \emph{semi-stable} (resp. \emph{stable}) if for every proper subbundle $\F\subset \E$, we have $\mu(\F) \leq \mu(\E)$ (resp. $\mu(\F) < \mu(\E)$). It was proven by Harder-Narasimhan \cite{hn} that over a Riemann surface, every bundle $\E$ admits a \emph{canonical} filtration by subbundles $$\{0\}= \E_0 \subset \E_1\subset ... \subset \E_n = \E$$ such that $\mu(\E_i) >\mu(\E_{i+1})$ and $\E_{i}/\E_{i-1}$ is semi-stable. Let $r_i$ and $d_i$ be the rank and the degree of $\E_i/\E_{i-1}$. The sequence $((r_1,d_1),...,(r_n,d_n))$ is called the \emph{Harder-Narasimhan type} or \emph{HN-type} of $\E$.

Let $E\rightarrow \Sigma$ be a smooth $\C^r$-bundle of degree $d$, and let $C(r,d)$ be the \emph{space of holomorphic structures} on $E$. Choosing a basepoint in $C(r,d)$ determines a diffeomorphism  $$ C(r,d) \cong \Omega^{0,1}(\Sigma, End(E)),$$ which is a contractible complex, Banach manifold after appropriate Sobolev completion (\cite{ab2} section 14). The complex gauge group $\G_{\C}(r,d)$ acts on $C(r,d)$, and there is a natural bijection of sets 

$$C(r,d)/\gau_{\C}(r,d)  \stackrel{1:1}{\longleftrightarrow} \frac{\{ \text{holomorphic bundles of rank r and degree d over $\Sigma$}\}}{\text{isomorphism}}  $$ 
Decomposing $C(r,d)$ according to Harder-Narasimhan types $\lambda = ((r_1,d_1),...,(r_n,d_n))$
produces an equivariant stratification\footnote{Stratification (\ref{laoecbulcob}) may also be interpreted as the Morse theoretic stable manifolds for the Yang-Mills functional. This point of view will not be used in this paper.}
\begin{equation}\label{laoecbulcob} C(r,d)  = \bigcup_{\lambda} C_{\lambda}(r,d) \end{equation}
into locally closed, finite codimension complex submanifolds, indexed by $\lambda$ satisfy $r_1+...+r_n = r$, $d_1+...+d_n = d$ and $d_1/r_1 > ....> d_n/r_n$.  The semi-stable stratum $C_{ss}(r,d):= C_{((r,d))}(r,d)$ is dense and open, and we have a surjective map

	$$ C_{ss}(r,d)/\G_{\C}(r,d) \twoheadrightarrow M_{\Sigma}(r,d) $$
	which is a homeomorphism when $gcd(r,d)=1$. Atiyah and Bott \cite{ab2} prove that the stratification (\ref{laoecbulcob}) is \emph{equivariantly perfect} for any coefficient field. We take a moment to explain this result.

Given a topological group $G$ and a $G$-space $X$, the equivariant Poincar\'e series of $X$ is the generating function $$P_t^G(X) = \sum_{i=0}^{\infty} \dim (H^i_G(X)) t^i$$ where $H^*_G(X)$ is the Borel equivariant cohomology of $X$ over some fixed coefficient field. The equivariant perfection result of Atiyah and Bott states that

\begin{equation}\label{orbulabocelb0} P_t^{\G_{\C}(r,d)}(C(r,d)) = \sum_{\lambda} t^{2d_{\lambda}}P_t^{\G_{\C}(r,d)}(C_{\lambda}(r,d)) \end{equation}
where $d_{\lambda}$ is the complex codimension of $C_{\lambda}(r,d)$ in $C(r,d)$. In other words, up to degree shifts, the equivariant Betti numbers of $C(r,d)$ is simply the sum of those of the strata.  Because $C(r,d)$ is contractible, it follows that 
\begin{equation}\label{orbulabocelb1}P_t^{\G_{\C}(r,d)}(C(r,d)) = P_t(B\G_{\C}(r,d)). \end{equation}
Furthermore, for an unstable stratum $\lambda = ((r_1,d_1),...,(r_n,d_n))$, Atiyah and Bott demonstrate that  
\begin{equation}\label{orbulabocelb2} P_t^{\G_{\C}(r,d)}(C_{\lambda}(r,d)) = \prod_{i=1}^n P_t^{\G_{\C}(r_i,d_i)}(C_{ss}(r_i,d_i)).
\end{equation}  Rearranging (\ref{orbulabocelb0}) and substituting (\ref{orbulabocelb1}) and (\ref{orbulabocelb2}) yields the formula
\begin{equation}\label{alroecbualrcoe}
P_t^{\G_{\C}(r,d)}(C_{ss}(r,d)) = P_t(B\G_{\C}(r,d)) - \sum_{\lambda \neq (r,d)} t^{2d_{\lambda}} \prod_{i=1}^n P_t^{\G_{\C}(r_i,d_i)}(C_{ss}(r_i,d_i))
\end{equation} 	
which recursively expresses $P_t^{\G_{\C}(r,d)}(C_{ss}(r,d))$ in terms of the lower rank cases $P_t^{\G_{\C}(r_i,d_i)}(C_{ss}(r_i,d_i))$. 
Finally, if $gcd(r,d) =1$ then $$P_t(M_{\Sigma}(r,d)) =  (1-t^2)P_t^{\G_{\C}(r,d)}(C_{ss}(r,d)). $$ The correction factor $(1-t^2) =  1/P_t(B\C^*)$ is due to the constant scalar action by $\C^*$ acting trivially on $C(r,d)$.

A parallel story can holds for real/quaternionic vector bundles. Given a real/quaternionic structure $\tau$ on a smooth $\C^{r}$-bundle of degree $d$, define 

\begin{itemize}
\item $C(r,d,\tau) \subset C(r,d)$,~~the \emph{space of real/quaternionic holormophic structures}
\item $\G_{\C}(r,d,\tau)\subset \G_{\C}(r,d)$,~~the \emph{real/quaternionic gauge group}
\end{itemize}
to be those operators commuting with $\tau$. Equivalently, $\tau$ determines involutions on $C(r,d)$ and $\G_{\C}(r,d)$ for which $C(r,d,\tau)=C(r,d)^{\tau}$ and $\G_{\C}(r,d,\tau) = \G_{\C}(r,d)^{\tau}$ are the fixed points. Define the \emph{moduli space of real/quaternionic semi-stable bundles} of type $\tau$,
$$ M(r,d,\tau) = M_{(\Sigma, \sigma)}(r,d,\tau) := C_{ss}(r,d,\tau)/ \G_{\C}(r,d,\tau) .$$
According to Schaffhauser \cite{schaffhauser2010real}, if $gcd(r,d) =1$, then we may identify $ M(r,d,\tau)$ with a corresponding path component of the set of real points $M(r,d)^{\sigma}$. 

In the current paper, we adapt the Atiyah-Bott method to derive recursive formulas for the $\Z/2$-Betti numbers of $M(r,d,\tau)$. We will focus on moduli spaces of real bundles, because quaternionic case reduces to the real case by the following remark. 

\begin{remark}\label{thanks} If $M(r,d,\tau)$ is a moduli space of quaternionic bundles with $gcd(r,d)=1$, then by Remark \ref{needremakr}, there exists a quaternionic line bundle $(L,\tau')$ of some degree $d'$. Tensor product by $(L,\tau')$ defines an isomorphism between $M(r,d,\tau)$ and the moduli space of real bundles $M(r, d + r d', \tau \otimes \tau')$ which also has coprime rank and degree.
\end{remark}

\subsection{Summary}

In \S \ref{sect2}, we construct a stratification into locally closed, finite codimension submanifolds 
$$C(r,d,\tau) = \bigcup_{\lambda} C_{\lambda}(r,d,\tau)$$ indexed by \emph{real Harder-Narasimhan} types $\lambda = ((r_1,d_1,\tau_1),...,(r_n,d_n,\tau_n))$, and prove that the stratification satisfies the conditions necessary to apply the standard Morse theory arguments.

In \S \ref{sect3}, we show that the stratification is $\G_{\C}(r,d,\tau)$-equivariantly perfect for $\Z/2$-coefficients.   This implies a recursive formula
\begin{equation}\label{aolrcubalocre} P_t^{\G_{\C}(r,d,\tau)}(C_{ss}(r,d,\tau))  =  P_t(B\G_{\C}(r,d,\tau)) - \sum_{\lambda \neq (r,d,\tau)} t^{d_{\lambda}} \prod_{i=1}^n P_t^{\G_{\C}(r_i,d_i\tau_i)}(C_{ss}(r_i,d_i,\tau_i)).
\end{equation}

Sections \S \ref{sect4}, \ref{sect5} and \ref{sect6} are devoted to calculating the Poincar\'e series $P_t(B\G_{\C}(r,d,\tau))$ which is needed as input for the recursive formula (\ref{aolrcubalocre}), and this calculation takes up the bulk of the paper. The calculations involve Eilenberg-Moore spectral sequences, which are reviewed in Appendix \ref{EM spec seq}. We find it convenient to work instead with the subgroup of unitary gauge transformations $ \G(r,d,\tau) \subseteq \G_{\C}(r,d,\tau)$, which includes as a deformation retract. 

In \S \ref{sect7} we prove that if gcd($r$,$d$) = 1 then $$P_t(M_{\Sigma}(r,d,\tau)) =  (1-t)P_t^{\G_{\C}(r,d)}(C_{ss}(r,d,\tau)), $$
where now the factor $(1-t) = (P_t(B\R^*))^{-1}$ corrects for a trivial scalar action by $\R^*$ on $C(r,d,\tau)$. Combined with the recursive formula (\ref{aolrcubalocre}) this allows a calculation of $P_t(M_{\Sigma}(r,d,\tau))$, and we present explicit formulas for ranks $r=1,2$ and $3$.

Throughout the paper, we make frequent reference to \cite{ab2} and we recommend that readers have a copy close at hand.

This paper covers largely the same ground as the independent paper by Liu-Schaffhauser \cite{liu2011yang}. The biggest difference in methods is that we use Eilenberg-Moore spectral sequences instead of Serre spectral sequences to compute the Poincare series $P_t(B\G_{\C}(r,d,\tau))$. Their paper also considers more directly the case of quaternionic bundles and solves the recursion (\ref{aolrcubalocre}) to get closed formulas for the Poincar\'e series $P_t(M_{\Sigma}(r,d,\tau)) $.  

\textbf{Notation:} For $G$ a topological group, $X$ a $G$-space, we denote the homotopy quotient $X_{hG} = EG \times_G X$.  We denote holomorphic bundles by $\E$ and $\D$ and the underlying $C^{\infty}$ or topological bundles by $E$ and $D$.

\textbf{Acknowledgements:}
I want to thank Jacques Hurtubise, Melissa Liu, Florent Schaffhauser, and all the participants at the Los Andes conference for valuable discussions, Misha Kotchetov for advice about Hopf algebras, and to the referee who made several useful recommendations including Remark \ref{thanks}. This research was supported by an NSERC Discovery grant.

\section{The Harder-Narasimhan stratification}\label{sect2}

\subsection{Harder-Narasimhan over complex curves}

We summarize the relevant material from Section 7 of \cite{ab2} that has not already been explain in \S \ref{intro}.

Let $\Sigma$ be a Riemann surface and $E \rightarrow \Sigma$ a smooth $\C^r$-bundle of degree $d$. Let $C(r,d) = C(E)$ denote the space of holomorphic structures on $E$ (under an appropriate Sobolev completion). For a given HN-type $\lambda = ( (r_1,d_1),...,(r_k,d_k))$, choose a $C^{\infty}$-splitting of $E = D_1 \oplus ... \oplus D_k$ where $(r_i, d_i)$ are the rank and degree of $D_i$ respectively. This determines an injective map 
\begin{equation}\label{loch}
	\prod_{i=1}^k C_{ss}(r_i,d_i) \hookrightarrow C_{\lambda}(r,d)
\end{equation}
that induces a homotopy equivalence of homotopy quotients 

\begin{equation}\label{loch2}
\prod_{i=1}^kC_{ss}(r_i,d_i)_{h\G_{\C}(r_i,d_i)} \cong C_{\lambda}(r,d)_{h\G_{\C}(r,d)},
\end{equation}
responsible for the equality of Poincar\'e series (\ref{orbulabocelb2}).

Each stratum $C_{\lambda}(r,d) \subseteq C(r,d)$ is a finite codimension submanifold with complex normal bundle $N_{\lambda} \rightarrow C_{\lambda}(r,d)$.  A holomorphic bundle  $\ca{E} \in \prod_{i=1}^k C_{ss}(r_i,d_i) \subseteq C_{\lambda}(r,d),$ decomposes as $ \E =  \D_1 \oplus ... \oplus \D_k$ and the normal bundle $N_{\lambda}$ of $C_{\lambda}(r,d)$ can be identified at $\E$ with 
\begin{equation}\label{lfd}
N_{\lambda, \E} \cong \bigoplus_{i<j} H^1(\Sigma, \D_i^*\otimes \D_j).
\end{equation}
The complex rank can be computed using Riemann-Roch and is given by the formula
\begin{equation}
d_{\lambda} := \rank_{\C} N_{\lambda} = \sum_{i<j} d_ir_j -d_jr_i +r_ir_j(g-1)	
\end{equation}
The points in the stratum $C_{\lambda}(r,d)$ are fixed by the subgroup  $G_{\lambda} \subset \G_{\C}(r,d)$ isomorphic to $(\C^*)^k$ that acts by scalar multiplication on the summands $E = D_1 \oplus ...\oplus D_k$.  However, $G_{\lambda}$ acts non-trivially on the normal bundle by $(t_1,...,t_k) \in G_{\lambda}$ multiplying the summand $H^1(\Sigma, \D_i^* \otimes \D_j)$ by the scalar $t_i^{-1} t_j$.

\subsubsection{Over $\C P^1$}\label{cp1complex}

For later use, we consider more explicitly the Harder-Narasimhan decomposition and the Atiyah-Bott formula in the special case $\Sigma = \C P^1$ where some simplifications occur.

By a result of Grothendieck \cite{grothendieck1957classification}, holomorphic bundles over $\C P^1$ are always isomorphic to a direct sum of line bundles. Consequently, every rank $r$ degree $d$ bundle must have the form $\mathcal{O}(k_1) \oplus ... \oplus \mathcal{O}(k_r)$ for some integers $k_1 \geq ...\geq k_r$ such that $k_1+...+k_r =d$. The corresponding stratum in $C(r,d)$ is a single $\G_{\C}(r,d)$-orbit with stabilizer isomorphic to $Aut( \mathcal{O}(k)^{\oplus r}) \cong GL_{r_1}(\C) \times ...\times GL_{r_n}(\C)$ where $r_1,...,r_n$ are the multiplicities of degrees occurring in the sequence $d_1 \geq ...\geq d_r$. The recursive formula (\ref{aolrcubalocre}) can be rewritten in this case as

\begin{equation}
P_t(B\G_{\C}(r,0)) = \sum_{\substack{k_1\geq ...\geq k_r \\
 k_1 +...+k_r =0}}   t^{2(\sum_{k_i>k_j} k_i -k_j -1 )} P_t( BAut( \oplus_{i=1}^r \mathcal{O}(k_i)))
\end{equation}

If any $k_i$ has absolute value greater than one, the index $2(\sum_{k_i>k_j} k_i -k_j -1 )$ is greater than $r$.  Consequently, in the stable limit $$B\G_{\C}(\infty, 0) := \lim_{r\rightarrow \infty} B\G_{\C}(r,0),$$ we only need to consider strata for which $|k_i| \leq 1$ for all $i$. In particular, we obtain the formula

\begin{eqnarray*} P_t(B\G_{\C}(\infty,0)) &=& \sum_{n=0}^{\infty} t^{2n^2} P_t(\lim_{r\rightarrow \infty} BAut(\mathcal{O}(1)^{\oplus n}\oplus \mathcal{O}^{\oplus r-2n}\oplus \mathcal{O}(-1)^{\oplus n})) \\
	&=& \sum_{n=0}^{\infty} t^{2n^2} P_t(BU_n)^2 P_t( \lim_{r\rightarrow \infty} BU_{r-2n})\\
	 &= &P_t(BU)\sum_{n=0}^{\infty} t^{2n^2} P_t(BU_n)^2 \end{eqnarray*}
Substituting known values on both sides of the equation produces the formula

\begin{equation}\label{darinking} \prod_{k=1}^{\infty}\frac{1}{ (1-t^{2k})^2} =  \Big( \prod_{k=1}^{\infty} \frac{1}{1-t^{2k}}\Big) \sum_{n=0}^{\infty}  t^{2n^2}\Big(\prod_{k=1}^n \frac{1}{(1-t^{2k})^2}\Big) . \end{equation}
Substituting $x = t^2$ and simplifying yields the formula
\begin{equation}\label{oeulracbo} \prod_{k=1}^{\infty}\frac{1}{ (1-x^{k})} =   \sum_{d=0}^{\infty} \frac{ x^{d^2}}{\prod_{k=1}^d (1-x^{k})^2} =   \sum_{d=0}^{\infty}\prod_{k=1}^d \frac{ x^{d}}{ (1-x^{k})^2}. \end{equation}

\begin{remark} Equation (\ref{oeulracbo}) also has a combinatorial proof. The left hand side of (\ref{oeulracbo}) is the generating function $\sum_{n=0}^{\infty} p(n) x^n$ where $p(n)$ counts partitions of $n$, or equivalently the number of Young diagrams of size $n$. The right hand side also counts partitions, where the $d$th term is the generating function counting Young diagrams containing a $d\times d$-square but no $(d+1)\times(d+1)$-square.
\end{remark}

\subsection{Harder-Narasimhan over real curves}\label{aorecbularcoj}

Let $M$ be a smooth manifold, possibly infinite dimensional and let $$ M  = \bigcup_{\lambda \in I} M_{\lambda}$$ be a partition of $M$ into locally closed, finite codimension submanifolds $M_{\lambda}$.  To apply the standard Morse-Bott arguments, the index set $I$ must admit a partial order $\leq$ satisfying the following properties  (see \cite{ab2} section 1).
\begin{enumerate}
\item For each $\lambda \in I$, the closure $ \overline{M_{\lambda}} \subseteq \bigcup_{\mu \geq \lambda} M_{\mu}$,
\item the complement of any finite subset of $I$ contains a finite number of minimal elements, and
\item for each integer $q$ there are only finitely many strata of codimension less than or equal to $q$.
\end{enumerate}
A stratification satisfying all of the above is said to satisfy the \emph{Morse package}.

Let $(E,\tau)$ denote a $C^{\infty}$-real\footnote{To simplify the exposition, we refer only to real bundles in this subsection.  The quaternionic case is almost identical}bundle over a real surface $(\Sigma,\sigma)$ of rank $r$ and degree $d$. Then $\tau$ induces an involution of $C(E)=C(r,d)$ and the set of fixed points $C(E)^{\tau} = C(E,\tau) = C(r,d,\tau)$ is an affine manifold modeled on $ \Omega^1( \Sigma, E)^{\tau}$.  Select $\E \in C(E, \tau)$. Because the involution $\tau$ respects the holomorphic structure of $\E$, it must also preserve the Harder-Narasimhan filtration $\E_0 \subset \E_1 \subset ... \subset \E_k =\E$. Consequently, the quotient bundles $\D_i = \E_i/\E_{i-1}$ are real  bundles. The list $((D_1,\tau_1),...,(D_k,\tau_k))$ of isomorphism types of $C^{\infty}$-real bundles is called the \emph{real Harder-Narasimhan type} of $(\E, \tau)$. 

\begin{prop}
 The affine manifold $C(r,d,\tau)$ admits a stratification into finite codimension, locally closed submanifolds

\begin{equation}\label{lcrchoelcruh}
 C(r,d, \tau) = \bigcup_{\lambda} C_{\lambda}(r,d, \tau)
 \end{equation}
indexed by real Harder-Narasimhan types $\lambda = ((D_1,\tau_1),...,(D_k,\tau_k))$ such that $(E,\tau) \cong (D_1\oplus...\oplus D_k, \tau_1\oplus ... \oplus \tau_k)  $ .  The stratification admits a partial order $\leq$ satisfying the Morse package.
\end{prop}

\begin{proof}

By results of Atiyah-Bott, the complex HN-stratification

\begin{equation}\label{osrcbicoa}
 C(r,d) =  \bigcup_{\mu} C_{\mu}(r,d)
 \end{equation}
satisfies the Morse package.  Consider the filtration induced on $C(r,d,\tau) \subset C(r,d)$ by intersecting with (\ref{osrcbicoa})

\begin{equation}\label{alcelurca}
C(r,d,\tau) = \bigcup_{\mu} (C_{\mu}(r,d) \cap C(r,d,\tau)).
\end{equation}
with the restricted partial order. Because  $C(r,d,\tau)$ is the fixed point set of a $\Z/2$-action preserving the stratification (\ref{osrcbicoa}), standard arguments from the theory of proper  group actions on manifolds tell us that (\ref{alcelurca}) inherits the Morse package. 

The decomposition (\ref{lcrchoelcruh}) is a refinement of (\ref{alcelurca}). Indeed for each complex HN-type $\mu$, we have a finite partition

$$ C_{\mu}(r,d) \cap C(r,d,\tau) = \bigcup_{f(\lambda) = \mu} C_{\lambda}(r,d,\tau)  $$
indexed by the real HN-types $\lambda$ that map to $\mu$ under the forgetful map $f$. Thus to complete the proof, it is enough to show that  each $C_{\lambda}(r,d,\tau) $ is a union of path-components of $ C_{\mu}(r,d) \cap C(r,d,\tau)$.  Let $\gamma: I \rightarrow (C_{\mu}(r,d)\cap C(r,d,\tau))$ be a path. Then for a fixed smooth real bundle $(E,\tau)$,  for each $t$ the holomorphic structure $\gamma(t)$ produces a continuously varying filtration of vector bundles  $E_1(t) \subset  E_2(t) ... \subset E$ preserved by $\tau$. Because the subbundle $E_i (t)$ varies continuously with $t$, we attain a $\tau$-subbundle
$$ F \subset \gamma^*E = E \times I $$
with $E_i(t) = F_t$. 
Applying the rigidity results of Palais-Stewart \cite{ps} to the sphere bundle of $F$, we find that $E_i(0)$ and $E_i(1)$ are isomorphic as $\Z/2$-equivariant smooth vector bundles.   Therefore $\gamma(0)$ and $\gamma(1)$ have the same real  HN-type.

\end{proof}

\begin{thm}\label{Aoeurcoler}
	For a given real HN-type $\lambda = ((r_1,d_1,\tau_1),...,(r_k,d_k, \tau_k))$, a choice of $C^{\infty}$-splitting of $(E,\tau) = (D_1 \oplus ... \oplus D_k, \tau_1 \oplus ...\oplus \tau_k)$ into real bundles determines a homotopy equivalence of homotopy quotients 
	\begin{equation}
	\prod_{i=1}^kC_{ss}(r_i,d_i, \tau_i)_{h\G_{\C}(r_i,d_i, \tau_i)} \cong C_{\lambda}(r,d,\tau)_{h\G_{\C}(r,d,\tau)}.
	\end{equation}
\end{thm}

\begin{proof}
This is proved exactly like (\cite{ab2} Prop. 7.12).
\end{proof}

For a point $$\E = (\D_1,...,\D_k) \in \prod_{i=1}^kC_{ss}(r_i,d_i, \tau_i) \subseteq C_{\lambda}(r,d,\tau)$$ the fibre of the normal bundle $N_{\lambda}^{\tau}$ is identified with 

\begin{equation}\label{alorcebualrco}
N_{\lambda,\E}^{\tau} = (\bigoplus_{i<j} H^1(\Sigma, \D_i^*\otimes \D_j))^\tau =  \bigoplus_{i<j} H^1(\Sigma, \D_i^*\otimes \D_j)^{\tau_i^* \otimes \tau_j}
\end{equation}
Let $G_{\lambda}^{\tau} \subset \G_{\C}(r,d,\tau)$ be the subgroup isomorphic to $(\R^*)^k$ that acts by scalar multiplication on the summands $D_i$. An element $(t_1,...,t_k)$ acts trivially on $\prod_{i=1}^kC_{ss}(r_i,d_i, \tau_i)$ and acts on acts on the normal bundle (\ref{alorcebualrco}) by multiplying the summand $H^1(\Sigma, \D_i^*\otimes \D_j))^\tau$ by $t_i^{-1}t_j$.

\section{Equivariant perfection}\label{sect3}

In the case of complex bundles, the basic topological result responsible for the equivariant perfection is the so-called Atiyah-Bott Lemma (\cite{ab2} Prop. 13.4). In our current situation, we require a variation on the Atiyah-Bott Lemma valid in characteristic 2.  A similar result, proven under more restrictive hypotheses, can be found in Goldin-Holm (\cite{goldin2004real}, Lemma 2.3).

\begin{lem}\label{aoelurcalorce}
Let $G$ be a compact connected Lie group with $H^*(G;\Z)$ torsion free. Let $X$ be a $G$-space of finite type and let $E\rightarrow X$ be a $G$-equivariant $\R^n$-vector bundle. Suppose that there exists $\epsilon \in G$ such that
\begin{itemize}
	\item $\epsilon^2$ is the identity in $G$
	\item $\epsilon$ acts trivially on $X$
	\item $\epsilon$ acts by scalar multiplication by $-1$ on $E$.
\end{itemize} 
Then the equivariant Euler class $Eul_G(X)$ is not a zero divisor in $H^*_G(X) =H^*_G(X;\Z_2)$.
\end{lem}

\begin{proof}
To begin, we reduce to the case that $G$ is abelian. Let $T \subset G$ be a maximal torus containing $\epsilon$, then by (\cite{ab2} 13.3) the functorial map $$ H^*_G(X) \rightarrow H^*_T(X) $$ is injective. Since the functorial map also sends $Eul_G(E)$ to $Eul_T(E)$, it suffices to show that $Eul_T(E)$ is not a zero divisor in $H^*_T(X)$.
	
Next, we reduce to the case of a circle group. Choose a decomposition $T \cong S^1 \times T'$, where $S$ is the circle group and $\epsilon = (-1,Id_{T'})$. Then there is a canonical isomorphism
$$ H_T^*(X) = H_{S^1}^*( X_{hT'}) $$
which identifies $Eul_T(E) = Eul_S(E_{hT'})$.  Thus, by replacing $X$ with $X_{hT'}$ and $E$ with $E'_{hT'}$,  it suffices to consider the case $G = S^1$ and $\epsilon = -1$.

So let $S = S^1$ and $C_2 = \{\pm 1\} \subset S$. 

\textbf{Claim:} The functorial map $ H^*_{S}(X) \rightarrow H^*_{C_2}(X)$ is injective. 

\begin{proof} The functorial map is induced $S/C_2$-principal fibration $\phi$
$$S/C_2 \rightarrow X_{hC_2}\stackrel{\phi} {\rightarrow}X_{hS}.$$
By considering the associated Gysin sequence we gain an inequality of Poincar\'e series
$$ P_t(X_{hC_2}) \leq P_t(X_{hS})(1+t)$$
with equality if and only if $\phi^*$ is injective. Since $C_2$ acts trivially on $X$ we have equality $P_t(X_{hC_2}) = P_t(X) P_t(BC_2) = P_t(X)/(1-t)$.  Furthermore, using the Serre spectral sequence of the fibration $X \rightarrow X_{hS} \rightarrow BS$ we get the inequality $ P_t(X_{hS})\leq P_t(X)P_t(BS) = P_t(X)/(1-t^2)$. Putting this all together we have
$$  P_t(X)/(1-t) = P_t(X_{hC_2})  \leq P_t(X_{hS})(1+t) \leq P_t(X)(1+t)/(1-t^2)= P_t(X)/(1-t) $$
so all of these inequalities are equalities and we are done.
\end{proof}

The injective map $H^*_S(X) \rightarrow H^*_{C_2}(X)$ sends $Eul_S(E)$ to $Eul_{C_2}(E)$, so it is enough to show that $Eul_{C_2}(E)$ is not a zero divisor in $H^*_{C_2}(X) = H^*(X)\otimes H^*(BC_2)$. This becomes a straight forward argument in direct analogy with the proof of (\cite{ab2} Prop. 13.4). This argument is carried out in Goldin-Holm (\cite{goldin2004real} Lemma 2.3), though they state the lemma with unnecessarily restrictive hypotheses suited to their applications in symplectic geometry.
\end{proof}

\begin{lem}\label{aolercuaroceb}
Consider a stratum $ C_{\lambda}(r,d, \tau)$ with $\lambda = ((D_1, \tau_1),...,(D_k,\tau_k))$ and normal bundle $N_{\lambda}^{\tau}$. 
Then $Eul_{\G_{\C}(r,d,\tau)}(N_{\lambda}^{\tau})$ is not a zero divisor in $H^*_{\G_{\C}(r,d,\tau)}(C_{\lambda}(r,d, \tau))$.
\end{lem}
\begin{proof}

For notational simplicity, denote $\G_i := \G_{\C}(r_i,d_i, \tau_i) $ and $C_i := C_{ss}(r_i,d_i, \tau_i) $.  As explained in \S \ref{aorecbularcoj}, we have a homotopy equivalence

\begin{equation}
\prod_{i=1}^k(C_i)_{h\G_i} \cong C_{\lambda}(r,d,\tau)_{h\G_{\C}(r,d,\tau)}.
\end{equation}
under which there is an isomorphism of vector bundles

\begin{equation}\label{aoeurblaoerc}
(N_{\lambda}^{\tau}|_{C_1 \times ... \times C_k})_{h(\G_1 \times ...\times \G_k)} \cong (N_{\lambda}^{\tau})_{h\G_{\C}(r,d,\tau)} .
\end{equation}
We can also form the vector bundle (\ref{aoeurblaoerc}) in two stages.  Let $p \in \Sigma$ be a point that is not fixed by $\sigma$, then we have short exact sequences
$$\G_i^{bas} \rightarrow \G_i \rightarrow GL(D_{i,p}) $$
where $\G_i^{bas} \subset \G_i$ is the subgroup that acts trivially on the fibre $D_{i,p}$ and $GL(D_{i,p})$ is the general linear group of the fibre.  Up to homotopy, we may restrict to the subgroup $U(D_{i,p}) \subset GL(D_{i,p})$. The subgroup $\G^{bas}_1 \times ...\times \G^{bas}_k $ is normal, so we can form the homotopy quotient in stages 
$$ (N_{\lambda}^{\tau}|_{C_1 \times ... \times C_k})_{h(\G_1 \times ...\times \G_k)} \cong  ((N_{\lambda}^{\tau}|_{C_1 \times ... \times C_k})_{h(\G_1^{bas} \times ...\times \G_k^{bas})})_{h(U(D_{1,p})\times ...\times U(D_{k,p}))}$$ 

The vector bundle $(N_{\lambda}^{\tau}|_{C_1 \times ... \times C_k})_{h(\G_1^{bas} \times ...\times \G_k^{bas})}$ decomposes into summands according to (\ref{alorcebualrco}). The central subgroup $\prod_{i=1}^k C_2^{i} \subset \prod_{i=1}^kU(D_{i,p})$ acts trivially on $(C_1 \times...\times C_k)_{h(\G_1^{bas} \times ...\times \G_k^{bas})} $ and $(t_1,...,t_k)\in \prod_{i=1}^k C_2^{i}$ acts on $N_{\lambda}$ by scalar multiplying the summand $H^1(\Sigma, Hom(D_i, D_j))^{\tau}$ by $t_i^{-1}t_j$. Applying Lemma \ref{aoelurcalorce},  we conclude that Euler classes of the summands of $(N_{\lambda}^{\tau}|_{C_1 \times ... \times C_k})_{h(\G_1 \times ...\times \G_k)}$ are not a zero-divisors, so $Eul((N_{\lambda}^{\tau}|_{C_1 \times ... \times C_k})_{h(\G_1 \times ...\times \G_k)})$ is not a zero divisor.
\end{proof}

\begin{thm}\label{erablricbeoajj}
For $(E,\tau)$ a real/quaternionic $C^{\infty}$-bundle over a real curve $(\Sigma,\sigma)$, the Harder-Narasimhan stratification of $C(E, \tau)$ is $\G(E,\tau)$-equivariantly perfect, establishing the recursive formula (\ref{aolrcubalocre}).
\end{thm}
\begin{proof}
This is follows from Lemma \ref{aolercuaroceb} by the self-completing principle of Atiyah-Bott (\cite{ab2} Prop. 1.9).	
\end{proof}

\section{Classifying spaces of gauge groups}\label{sect4}

Let $G$ be a topological group and $P \rightarrow M$ a principal bundle over a finite cell complex $M$.  Let $$\G(P) = \G_P = Maps_G(P, G)$$ denote the group of continuous gauge transformations. If $BG$ can be represented by a CW-complex (say if $G$ is a Lie group), then there is a homotopy equivalence (see Atiyah-Bott \cite{ab2} Prop. 2.4)
\begin{equation}\label{erceb} B\G(P) \cong Maps_P( M, BG) \end{equation}
where $Maps( M, BG) $ is the space of continuous maps from $M$ to $BG$ with compact-open topology, and $Maps_P( M, BG) $ is the path component classifying $P$.

Given a $\C^r$-vector bundle $E$, we denote by $\G_{\C}(E)$ the gauge group of the $GL_r(\C)$-frame bundle of $E$ and by $\G(E)$ the gauge group of the orthonormal frame bundle with respect to an unspecified Hermitian metric.  It is explained in (\cite{ab2} section 8), the natural inclusion $\G(E) \hookrightarrow \G_{\C}(E)$ is  is a homotopy equivalence, so they are largely interchangeable for our purposes. We prefer to work with  $\G(E)$ to take advantage of the compactness of $P$.

Suppose that $f: N \rightarrow M$ is a continuous map of finite complexes, and $\phi: G \rightarrow H$ a homomorphism of topological groups. Combining pull-back and induction (in either order), form the $H$-bundle $f^*P \times_G H$ over $N$. There is a canonically induced homomorphism of gauge groups $\psi: \G(P) \rightarrow \G(f^*P \times_G H)$.

\begin{prop}
Denote by $P' := f^*P \times_G H$. The following diagram commutes up to homotopy

\begin{equation}\label{aoleruccorei}
	\xymatrix{  B\G(P)  \ar[r]^{B\psi} \ar[d] & B\G(P') \ar[d] \\
              Maps_P(M, BG)  \ar[r] &   Map_{P'}(N, BH) }  
\end{equation}
where $B\psi$ is functorially induced by $\psi$, the vertical arrows are the isomorphism from \ref{erceb} and the bottom arrow is defined by composition by $f$ and $B\phi$. 
\end{prop}
\begin{proof}
We use the Milnor join construction of classifying spaces to make $B$ a functor \cite{milnor1956construction2}.  This construction models $EG$ as the infinite join $G^{* \infty}$. From this point of view, diagram (\ref{aoleruccorei}) is the orbit space map of the equivariant diagram

$$\xymatrix{  Maps_G(P, G)^{* \infty}  \ar[r] \ar[d] & Maps_{H}(P', H)^{* \infty} \ar[d] \\
              Maps_G(P, G^{* \infty})  \ar[r] &   Map_{H}(P', H^{* \infty}) }  $$
which is readily seen to be commutative on the nose.
\end{proof}

Using the identification (\ref{erceb}), we have an evaluation map
$$ev: M \times B\G(P) \rightarrow BG.$$
Define a linear map $t: H_p(M) \otimes H^q(BG) \rightarrow H^{q-p}(B\G(P))$ by 

\begin{equation}\label{rebairceb} t( \sigma \otimes \alpha) = \int_{\sigma} ev^*(\alpha) \end{equation}
where we the integral denotes the \emph{slant product} of $\alpha$ with respect to $\sigma$.

\begin{prop}\label{aeoucgxalbrcu}
Denote by $P' := f^*P \times_G H$ as before. The diagram  $$ \xymatrix{  H_*(N) \otimes H^*(BH) \ar[rr]^t \ar[d]^{f_* \otimes B\phi^*} & &  H^*(B\G_{P'}) \ar[d]^{B\psi^*} \\  H_*(M) \otimes H^*(BG) \ar[rr]^t &&   H^*(B\G( P))  } $$ commutes. In other words, $t$ is natural with respect to pull-back and induction of principal bundles.
\end{prop}

\begin{proof}
The square above factors as two squares that are both well known to commute

$$  \xymatrix{  H_*(N) \otimes H^*(BH) \ar[rr]^{Id \otimes ev^*} \ar[d]^{f_* \otimes B\phi^*} &&  H_*(N) \otimes H^*(N \times B\G_{\tilde{P}}) \ar[rr]^{\int} \ar[d]^{f_* \otimes B\psi^*}&&  H^*(B\G_{\tilde{P}}) \ar[d]^{B\psi^*} \\  H_*(M) \otimes H^*(BG) \ar[rr]^{Id \otimes ev^*} && H_*(M) \otimes H^*(M \times B\G_{P}) \ar[rr]^{\int}&&   H^*(B\G_P)  }. $$

\end{proof}

\subsection{Loop groups}

Given a Lie group $G$, the \emph{loop group} $LG = Maps(S^1, G)$ can be thought of as the group of gauge transformations of the trivial $G$ bundle over $S^1$.  By (\ref{erceb}),  we identify,
$$  BLG \cong  L_0 BG $$
where $L_0BG$ is the path component of $LBG = Maps(S^1, BG)$ containing the constant maps. Consider the fibration sequence  
\begin{equation}\label{oleuablcroeb}
\Omega BG \longrightarrow LBG \stackrel{ev_1}{\longrightarrow} BG
\end{equation}
where $ev_1$ is evaluation at the basepoint $1 \in S^1$.

\begin{prop}\label{aolercubaloxe}
	In case $G = U_r, SU_r$ or $O_r$, the fibre of (\ref{oleuablcroeb}) is totally non-homologous to zero in characteristic $2$.  Consequently, there are isomorphisms
	$$ H^*(LBG) \cong H^*(G) \otimes H^*(BG) $$
as graded $H^*(BG)$-modules.
\end{prop} 

\begin{proof}
We consider the case $G = O_r$ (cases $G = U_r$ and $G =SU_r$ are similar). Let $O = \lim_{r \rightarrow \infty} O_r$ denote the infinite orthogonal group.  By Bott Periodicity,  $BO$ is a loop space hence has the homotopy type of a topological group by a result of Milnor \cite{milnor1956construction}. Exploiting multiplication on $BO$, one easily constructs a trivialization of the bundle $$LBO \cong BO \times \Omega BO \sim BO \times O.$$ The inclusion $O_r \rightarrow O $ is surjective on $\Z_2$-cohomology so the morphism of fibration sequences
$$\xymatrix{ O \ar[r]  & LBO \ar[r] & BO  \\
           O_r \ar[r] \ar[u] & LBO_r \ar[r] \ar[u] & BO_r \ar[u] } $$
implies that the fibre inclusion $O_r \rightarrow LBO_r$ is a cohomology surjection and $\pi_1(BO_r) $ acts trivially on $H^*(O_r)$.  The result now follows from the Leray-Hirsch theorem.
\end{proof}

For the following Lemma, let $M = \vee_{i=1}^m S^1_i$ be a wedge of $m$ circles.  For some $p$, $0 \leq p \leq m$ let $\G$ be the subgroup of $Maps(M, U_r)$ of maps that restrict to contractible loops on the first $p$ circles. We have a composition of maps $B\G \rightarrow BMaps(M,U_r) = Maps(M, BU_r)$, so it makes sense to define an evaluation map
$$ ev: M \times B\G \rightarrow BU_r $$
and the operator $t: H_*(M)\otimes H^*(BU_r) \rightarrow H^*(B\G)$ as in Proposition \ref{aeoucgxalbrcu}.

Recall that $H^*(BU_r;\Z_2) = S(c_1,...,c_r)$ where $c_k$ is (the mod $2$ reduction of) the universal $k$th Chern class, with degree $|c_k| = 2k$. 

\begin{lem}\label{oaarceublrc} 
The cohomology ring $H^*( B\G)$ decomposes as a tensor product of a polynomial algebra generated by classes $c_k :=  t ([pt] \otimes c_{k}  )$ for $k=1,...,r$ and an exterior algebra generated by classes $\bar{c}_{i,k} = t ([S^1_i]\otimes c_{k}  )$ for $ (i,k) \in \{ \{1,..,m\}\times \{1,...,r\} | k\neq 1 \text{ if } i\leq p\}$
\end{lem}

\begin{proof}
We make use of a similar result stated for surface gauge groups and integral coefficients from \cite{ab2} Prop. 2.20. 

Restriction to the base point determines a fibration sequence
$$ (B\Omega_0U_r)^p \times (B\Omega U_r)^{m-p} \rightarrow B\G \rightarrow BU_r $$
where we have homotopy equivalences $B\Omega U_r \cong U_r$ and $B\Omega_0 U_r \cong SU_r$.  By the Leray-Hirsch theorem, it suffices to show that the classes $\bar{c}_{i,k}$ generate an exterior algebra that restricts to an isomorphism to the cohomology of the fibre.  Indeed, the inclusion $B\G \rightarrow BMaps(M,U_r)$ is a cohomology surjection, so it is enough to establish the case $p=0$.
Choose an embedding of  $M \hookrightarrow \Sigma$ as a retract in a closed surface (which must have genus at least $2m$). This induces an inclusion map $BMaps(M, U_r) \rightarrow BMaps(\Sigma,U_r)$ as a retract and thus a cohomology surjection.  The classes $\bar{c}_{i,k}$  are identified with the image of the classes $b_k^i$ of \cite{ab2} according to the functoriality of Proposition \ref{aeoucgxalbrcu}, so they form an exterior algebra that restricts isomorphically to the fibres.  \end{proof}

\section{Real gauge groups}\label{sect5}

Let $(M,\sigma)$ be a finite cell complex $M$ equipped with an automorphism $\sigma \in Aut(M)$ such that $\sigma^2= Id_M$.  A topological real vector bundle $(E,\tau)$ over $(M,\sigma)$ consists of a $\C^r$-vector bundle $\pi: E \rightarrow M$ and an antilinear bundle involution $E \rightarrow E$ such that $\tau^2 = Id_E$ and $\pi \circ \tau = \sigma \circ \pi$.

\begin{definition}
Given a real bundle $(E,\tau)$, the \emph{real gauge group} is defined $$\G_{\C}(E, \tau)  = \{ g \in \G_{\C}(E) |~g\tau = \tau g \}.$$
\end{definition}
We prefer to work with the unitary version of real gauge groups.  Fix a Hermitian metric on $E$ that is compatible with $\tau$ in the sense that orthonormal frames are sent to orthonormal frames.  Then we define
$$ \G(E, \tau)  = \G(E) \cap  \G_{\C}(E, \tau).$$
The inclusion $ \G(E, \tau) \hookrightarrow \G_{\C}(E, \tau)$ is a homotopy equivalence, because the coset space $  \G_{\C}(E, \tau) /  \G(E, \tau)$  can be identified with the convex space of $\tau$-compatible Hermitian metrics.  Thus for our purposes $ \G(E, \tau)$ and $\G_{\C}(E, \tau)$ are interchangeable.

The conjugation action $\Z/2 \curvearrowright U_r$, sending a matrix $[a_{i,j}]$ to $[\overline{a_{i,j}}]$ induces an involution on $BU_r$. Given a $\Z/2$-space $(X,\sigma)$, consider the space $Maps^{\Z/2}(X, BU_r)$ of equivariant maps. 

\begin{prop}\label{lracldio}
Isomorphism classes of topological real bundles $(E,\tau)$ over a finite $\Z/2$-cell complex $(X,\sigma)$ are classified by $\pi_0(Maps^{\Z/2}(X, BU_r))$. The classifying space $B\G_E^{\tau}$ is identified with the path component $Maps^{\Z/2}_E(X, BU_r)$ classifying $(E,\tau)$.
\end{prop}

\begin{proof}

The classification of isomorphism classes of bundles by $\pi_0(Maps^{\Z/2}(X, BU_r))$ is proven in \cite{biswas2009moduli} section 4, so we concentrate on the second statement.  

Let $(E,\tau) \rightarrow (X,\sigma)$ be a fixed topological real bundle, let $P \rightarrow X$ denote the unitary frame bundle, and let $\hat{U}_r = U_r \rtimes \Z/2$ be the semidirect defined by complex conjugation on $U_r$.  Then there is a natural identification
 $$\G_E^{\tau} \cong Maps_{\hat{U}_r}( P, U_r)$$
with the equivariant maps from $P$ to $U_r$.  If we represent $EU_r$ by the Milnor join construction, then $EU_r$ acquires a $\hat{U}_r$ action, and the space $Maps_{\hat{U}_r}( P, EU_r)$ forms a $\G_E^{\tau}$-bundle in a natural way, such that the orbit space $Maps_{\hat{U}_r}( P, EU_r) / \G_E^{\tau}$ is identified with the component of $Maps^{\Z/2}(X, BU_r)$ classifying $(E,\tau)$.

It remains to prove that $Maps_{\hat{U}_r}( P, EU_r)$ is contractible. We adapt an argument of Dold (\cite{dold1963partitions} section 8).  Recall that Milnor constructs $EU_r$ as the direct limit $\lim_{\rightarrow} U_r^{* n}$, where $U_r^{* n}$ denotes the n-fold join of $U_r$.  Because $P$ is a compact cell complex, it follows that
$$Maps_{\hat{U}_r}( P, EU_r) = \lim_{\rightarrow}Maps_{\hat{U}_r}( P, U_r^{* n}). $$
To prove that $Maps_{\hat{U}_r}( P, EU_r)$ is contractible, it suffices to show that for all $n$ there is some $m$ such that the inclusion

\begin{equation}\label{recloh}
Maps_{\hat{U}_r}( P, U_r^{* n}) \hookrightarrow Maps_{\hat{U}_r}(P, U_r^{* (m+n)})
\end{equation} 
is null-homotopic. The map (\ref{recloh}) factors through the inclusion 
$$ Maps_{\hat{U}_r}( P, U_r^{* n})  \stackrel{i}{\rightarrow}  Maps_{\hat{U}_r}( P, U_r^{* n}) * Maps_{\hat{U}_r}( P, U_r^{* m})  $$
and for any non-vacuous spaces $X$ and $Y$, the inclusion $X \rightarrow X * Y$ is null-homotopic, completing the proof. An explicit contraction can be constructed along the lines of \cite{dold1963partitions}.

\end{proof}

\subsection{Real loop groups}\label{orlecbularc}

A \emph{real loop group} is simply a real gauge group for a real bundle $(E,\tau)$ over $(S^1,\sigma)$ where $\sigma: S^1 \rightarrow S^1$ is an involution.  We consider two cases:   $\sigma = Id_{S^1}$  the identity map and $\sigma = -Id_{S^1}$ the antipodal map. As usual, we work with the Hermitian version $LU_r^{\tau} \subset LGL_r(\C)^{\tau}$.

\begin{prop}\label{eloced}
For any positive rank $r$, there are two isomorphism classes of topological real $\C^r$-bundles over $(S^1, Id_{S^1})$.  They are classified by the first Stieffel-Whitney number $w_1(E^{\tau}) \in H^1(S^1;\Z/2) = \Z/2$. 
\end{prop}  

\begin{proof}
Equivariant maps from $(S^1, Id)$ to $BU_r$ are the same thing as maps $S^1$ to $BO_r \subset BU_r$. Up to homotopy, these are in correspondence with $\pi_1(BO_r) = \Z/2$ and correspond to a choice of first Stieffel-Whitney class.
\end{proof}

\begin{prop}\label{arclricdo}
For any positive rank $r$, there is only one topological real bundle over $(S^1, -Id_{S^1})$ up to isomorphism.
\end{prop}

\begin{proof}
Any equivariant map from $S^1$ to $BU_r$ can be equivariantly contracted to a point (see \cite{biswas2009moduli} section 4.1).  
\end{proof}

\begin{remark}\label{alrocebucrao}
The path components $[\gamma] \in \pi_0(LU_r)$ are classified by the winding number of the map $$S^1 \rightarrow U(1),~~~~~~~~\theta \mapsto \det(\gamma(\theta)).$$ It is easily checked that for the examples above, $LU^{\tau}_r$ is contained in the identity component $L_0U_r \subset LU_r$.
\end{remark}

\subsection{Cohomology of real loop groups}

In this section we compute the $\Z/2$-Betti numbers of real loop groups $BLU_r^{\tau}$ and describe the map $$i^*: H^*(BLU_r)\rightarrow H^*(BLU_r^{\tau})$$ induced by inclusion. Recall from Lemma \ref{oaarceublrc} that $H^*(BLU_r) \cong \wedge(\bar{c}_1,...,\bar{c}_r) \otimes S(c_1,...,c_r)$. The main takeaway is the following corollary.

\begin{cor}\label{orecubrobeui}
For the real loop groups described in Proposition \ref{eloced} and \ref{arclricdo},  $H^*(BLU_r^{\tau})$ is a free $H^*(BU_r) =S(c_1,...,c_r)$ module on which $\wedge(\bar{c}_1,...,\bar{c}_r)$ acts trivially. The Poincar\'e series satisfy
$$ P_t( BLU_r^{\tau}) = \frac{1}{1+t^r} \prod_{k=1}^r\frac{(1+t^k)^2}{1-t^{2k}} $$
for $\sigma = Id_{S^1}$ independently of $\tau$, and
 $$  P_t( BLU_r^{\tau}) =         \prod_{k=1}^r \frac{1+t^{2k-1}}{1-t^{2k}} $$
  for $\sigma = -Id_{S^1}$.
\end{cor}

\begin{proof}
An immediate consequence of Propositions \ref{robelurcbe} and \ref{ealrcboeruc}.	
\end{proof}

\subsubsection{The case $\sigma = Id_{S^1}$}

\begin{prop}\label{robelurcbe}

Let $LU_r^{\tau}$ be a real loop group over $(S^1, Id_{S^1})$. Then
\begin{equation}\label{oeluarcbo}
H^*(BLU_r^{\tau})\cong  H^*(SO_r) \otimes S(w_1,...,w_r)
\end{equation}
with degrees $|w_k| =k$, as a graded free module over $S(w_1,...,w_r)$.  The inclusion induced map $i: BLU_r^{\tau} \rightarrow BLU_r$ satisfies $i^*(\bar{c}_k)= 0$ and $i^*(c_k)=w_k^{2}$. 
\end{prop}

\begin{proof}[Proof of Proposition \ref{robelurcbe}]
In this case $\sigma$ acts trivially on $S^1$, so $BLU_r^{\tau}$ may be identified with one of the two path components of $Maps^{\Z/2}(S^1, BU_r) = Maps(S^1, BO_r) = LBO_r$, so (\ref{oeluarcbo}) follows immediately from Proposition \ref{aolercubaloxe} where the $w_i$ are the Stieffel-Whitney classes .

To study $i^*$, we have $i^*(c_k) = w_k^2$ (Milnor-Stasheff \cite{milnor1974cc} problem 15A) and by Proposition \ref{aeoucgxalbrcu}
$$ i^*(\bar{c}_k) = i^*(t( [S^1]\otimes c_k)) = t( [S^1]\otimes w_k^2 ) = 2 \bar{w}_k w_k =0.$$ 
\end{proof}

\subsubsection{The case $\sigma = -Id_{S^1}$}

We begin with a lemma.  We call a fibration $F \rightarrow E \rightarrow B$ \emph{cohomologically trivial} if $\pi_1(B)$ acts trivially on $H^*(F)$ and the Serre spectral sequence collapses so $H^*(E) \cong H^*(B)\otimes H^*(F)$ as a graded $H^*(B)$-module.

\begin{lem}\label{ebarecbiro}
Let $f: B' \rightarrow B$ be a continuous map of path-connected spaces for which $f^*: H^*(B)\rightarrow H^*(B')$ is injective and let $F \rightarrow E \rightarrow B$ be a Serre fibration with $\pi_1(B)$ acting trivially on $H^*(F)$. Then $E$ is cohomologically trivial if and only if the pull-back $f^*E$ is cohomologically trivial.
\end{lem}

\begin{proof}
That the pull-back of a cohomologically trivial fibration is cohomologically trivial is an easy consequence of the Leray-Hirsch Theorem.  In the other direction, the injectivity of $f^*$ implies that $f$ induces a morphism of Serre spectral sequences which at the $E_2$-page is the injective map $$ f^* \otimes id_{H^*(F)}: H^*(B)\otimes H^*(F) \rightarrow H^*(B') \otimes H^*(F).$$  Thus if the spectral sequence for $f^*E$ collapses, then the spectral sequence for $E$ must as well.
\end{proof}

\begin{prop}\label{ealrcboeruc}
Let $LU_r^{\tau}$ be a real loop group of rank $r$ over $(S^1,- Id_{S^1})$. There is an isomorphism of $H^*(BU_r)$-modules $$H^*(BLU_r^{\tau}) \cong H^*(U_r) \otimes S(c_1,...,c_r)$$ with degrees $|c_k| = 2k$. The inclusion induced map satisfies $i^*(\bar{c}_k)=0$ and $i^*(c_k) = c_k$.
\end{prop}

\begin{proof}

Consider the fibration
\begin{equation}\label{alorecubl}
B \Omega U_r \rightarrow BLU_r^{\tau}  \rightarrow BU_r 
\end{equation} 
induced by evaluation at the basepoint $1 \in S^1$. Let $i_1: H \hookrightarrow LU_r^{\tau}$ be the subgroup sending the base point $1 \in S^1$ to $O_r$. Then evaluation at $1$ defines a fibration   
\begin{equation}\label{lorcebualcbo}
B\Omega U_r \rightarrow BH \stackrel{\pi}{\rightarrow} BO_r
\end{equation}
that is a pull-back of (\ref{alorecubl}) under the inclusion 
\begin{equation}\label{alrcedijng}
BO_r \rightarrow BU_r.
\end{equation}
On the other hand, because an element of $LU_r^{\tau} \subset Maps(S^1, U_r)$ is determined by its values on one half of $S^1$, and the elements of $H$ send $\pm 1 \in S^1$ to the same value in $O_r$, there is a second injection $i_2: H \hookrightarrow LU_r$ defined by $i_2(\gamma)(e^{i \theta}) = \gamma(e^{i \theta/2} )$ for $\theta \in [0,2\pi]$, producing (\ref{lorcebualcbo}) as a pull-back of (\ref{oleuablcroeb}) under base map (\ref{alrcedijng}).  Since (\ref{alrcedijng}) is a cohomology injection, and (\ref{oleuablcroeb})  is cohomologically trivial, the result follows from two applications of Lemma \ref{ebarecbiro}. Finally, we have a commutative diagram
$$ \xymatrix{ BLU_r^{\tau} \ar[r]^i & BLU_r\\
              BH \ar[r]^{Bi_2} \ar[u]^{Bi_1} &  BLU_r \ar[u]^{f}    }  $$
where $f$ is induced by a degree two map $ S^1 \rightarrow S^1$. By Proposition \ref{aeoucgxalbrcu}, $f^*(c_k) = c_k$ and $f^*(\bar{c}_k) = 2\bar{c}_k = 0$.  Since both $Bi_1^*$ and $Bi_2^*$ are injective, $i^*(c_k) = c_k$ and $i^*(\bar{c}_k)=0$. 
\end{proof}

\section{Real gauge groups over surfaces}\label{ealbleia}\label{sect6}

This entire section is devoted to proving

\begin{thm}\label{alcgjbrcd.}
Suppose $(\Sigma, \sigma)$ is a genus $g$ surface with real points $\Sigma^{\sigma} \cong \sqcup_{a} S^1$ and let $(E,\tau) \rightarrow (\Sigma, \sigma)$ be a real bundle of rank $r$ and degree $d$. Then the Poincar\'e series of the $B\G(r,d,\tau)$ satisfies
$$P_t(B\G(r,d,\tau)) = \frac{1-t^{2r}}{(1+t^r)^{a}} \prod_{k=1}^r \frac{(1+t^k)^{2a} (1+t^{2k-1})^{g+1-a}}{(1-t^{2k})^2} .$$
\end{thm}

\subsection{Constructing the real gauge group}
We use models of real surfaces that are slightly different from \cite{biswas2009moduli}.  Let $\Sigma_h = \Sigma_h(\hat{g},n)$ denote a genus $\hat{g}$ surface with $n$ disks removed, and boundary circles numbered from $1$ to $n$: $$\partial \Sigma_h \cong \coprod_{i=1}^n S^1_i.$$ 
Observe that $ \Sigma_h(\hat{g},n)$ is homotopy equivalent to a wedge of $2 \hat{g} + n-1$ circles. 

Given an $n$-tuple of real loop groups $(LU_r^{\tau_1},...,LU_r^{\tau_n})$, define $\G(\hat{g},n, r; \tau_1,...,\tau_n)$ by the pull-back diagram of groups

\begin{equation}\label{alorecbulrcoeb}
\xymatrix{ \G(\hat{g},n, r; \tau_1,...,\tau_n) \ar[r] \ar[d] & Maps(\Sigma(\hat{g},n), U_r) \ar[d]^{\pi}\\
\prod_{i=1}^n LU_r^{\tau_i} \ar[r] &  \prod_{i=1}^n LU_r  }
\end{equation}
where $\pi$ is induced by restriction to the boundary circles.  For technical reasons, we prefer to work with the identity component subgroups $L_0U_r \subseteq LU_r$ and this poses no problem by Remark \ref{alrocebucrao}.  Let $Maps_0(\Sigma(\hat{g},n), U_r)$ denote the subgroup of maps that restrict to contractible loops on the boundary circles. Then we have a pull-back diagram of groups

\begin{equation}\label{aoeurclorce...}
\xymatrix{ \G(\hat{g},n, r; \tau_1,...,\tau_n) \ar[r] \ar[d] & Maps_0(\Sigma(\hat{g},n), U_r) \ar[d]^{\pi}\\
\prod_{i=1}^n LU_r^{\tau_i} \ar[r] &  \prod_{i=1}^n L_0U_r  }
\end{equation}
for which $\pi$ is surjective.

\begin{prop}
Let $(\Sigma, \sigma)$ be a real curve with $\sigma$ orientation reversing, and let $(E, \tau) \rightarrow (\Sigma,\sigma)$ be a real bundle of rank $r$.  Then the real gauge group $\G(E,\tau)$ is isomorphic to $ \G(\hat{g},n, r; \tau_1,...,\tau_n)$ for some choice of $\hat{g}$, $n$, and $\tau_i$. 
\end{prop}

\begin{proof}

Suppose that $\Sigma$ has genus $g$ and the fixed point set $\Sigma^{\sigma}$ consists of 
$a\geq 0$ circles.  Then by the classification of real curves  (found in \cite{biswas2009moduli}, section 2),  
$(\Sigma, \sigma)$ is equivariantly homeomorphic to a quotient $(\Sigma_h(\hat{g}, n)  \times \{0,1\})/ \sim $ 
with involution $\sigma$ sending $(\theta,j)$ to $(\theta, j+1~mod~2)$. Here $2 \hat{g} +n-1 = g$ , and the quotient relation is defined on boundary circles by  $ (\theta, 0) \sim (\theta,1)$  if $ i \leq a$ and $(\theta ,0) \sim (\theta +\pi,1)$ if $i > a$,  where $a < n$ if $\Sigma \setminus \Sigma^{\sigma}$ is connected and  $a= n$ if not.

Finally, since the involution transposes the two copies of $\Sigma_h(\hat{g},n)$,  and the restriction of $E$ to one copy of $\Sigma_h(\hat{g},n) $ is trivial, we may identify $\G(E,\tau)$ with the subgroup of $Maps(\Sigma(\hat{g},n), U_r)$ satisfying the boundary conditions of lying in the appropriate real loop groups, determined by restricting $(E,\tau)$ to the boundary circles of $\Sigma_h(\hat{g},n) $. 
 \end{proof}

\subsection{The first spectral sequence}

In this section, we use the pull-back diagram (\ref{aoeurclorce...}) to compute the Betti numbers of $B\G^{\tau}$. It is convenient to first consider an auxiliary space. Denote by $X$ the surface $\Sigma(\hat{g},n)$ with an open disk removed and denote by $S \subseteq \partial X$ the newly introduced boundary circle. Note that $X$ is homeomorphic to $\Sigma(\hat{g},n+1)$, but the new boundary circle will play a different role than the others. Consider the pull-back diagram of topological groups

\begin{equation}\label{lrcedlrilde}
\xymatrix{ \tilde{\G}(\hat{g}, n, r; \tau_1,...,\tau_n) \ar[r] \ar[d] & Maps_0(X, U_r) \ar[d]^{\pi}\\
\prod_{i=1}^n LU_r^{\tau_i} \ar[r] &  \prod_{i=1}^n L_0U_r  }
\end{equation}
where $Maps_0(X, U_r)$ is the subgroup of $Maps(X, U_r)$ of maps sending all boundary circles to contractible loops in $U_r$. 

\begin{lem}\label{oeracboel}
Suppose that $\sigma_i = Id_{S^1}$ for the first $a$ boundary circles and the remaining $\sigma_i =-Id_{S^1}$. Then $B \tilde{\G}(\hat{g}, n, r; \tau_1,...,\tau_n)$ has $\Z/2$ Poincar\'e series
	$$  P_t(B \tilde{\G}(\hat{g}, n, r; \tau_1,...,\tau_n)) = \frac{1}{(1+t^r)^{a}} \prod_{k=1}^r \frac{(1+t^k)^{2a} (1+t^{2k-1})^{2 \hat{g} +n-a}}{1-t^{2k}}. $$
	\end{lem}
\begin{proof}

Applying the classifying space functor to (\ref{lrcedlrilde}) results in a pull-back diagram

\begin{equation}\label{dgaonpn2}
\xymatrix{ B\tilde{\G}(\hat{g}, n, r; \tau_1,...,\tau_n) \ar[r] \ar[d] & B Maps_0(X, U_r) \ar[d]^{\pi}\\
\prod_{i=1}^n BLU_r^{\tau_i} \ar[r] & \prod_{i=1}^n BL_0U_r  }
\end{equation}
We calculate the Betti numbers of $B\tilde{\G}(\hat{g}, n, r; \tau_1,...,\tau_n)$ using an Eilenberg-Moore spectral sequence (EMSS). We review the EMSS in Appendix \ref{aolrecubque}.  

Let $R := H^*(\prod_{i=1}^n BL_0U_r)$ ($\Z/2$ coefficients understood throughout). The EMSS associated to (\ref{dgaonpn2}) converges to $H^*(B\tilde{\G}(\hat{g}, n, r; \tau_1,...,\tau_n))$ and has second page equal to the bi-graded algebra

\begin{equation}\label{Oeularocebr}
 EM_2^{*,*} =  Tor_R^{*,*}( H^*(\prod_{i=1}^n BLU_r^{\tau_i}) ,  H^*(BMaps_0(X, U_r)) ).
\end{equation}

For the rest of this section we use index sets, $i \in \{1,...,n\}$, $i' \in \{2,...,n\}$ $k \in \{1,...,r \}$, and $k' \in \{2,...,r\}$.  We use the notational convention that the appearance of one of these subscripts means to include the full range of that index set. 

Applying Lemma \ref{oaarceublrc} and the Kunneth theorem

$$R := \bigotimes_{i=1}^nH^*( BL_0U_r) \cong \wedge( \bar{c}_{i,k'}) \otimes S(c_{i,k})$$
where $|\bar{c}_{i,k}|= 2k-1, |c_{i,k}| =2k$.

\begin{lem}\label{aolercubalcroe}
There is an isomorphism
$$H^*( B Maps_0(X, U_r)) \cong \wedge(\bar{c}_{i,k'}) \otimes S(c_{k}) \otimes A, $$
where $A$ is an exterior algebra with Poincar\'e series $$P_t(A) = \prod_{k=1}^r (1+t^{2k-1})^{2 \hat{g}}.$$ 
In these generators, the bundle map $\pi^*: R \rightarrow H^*(Maps_0(X, U_r))$ satisfies  $\pi^*(\bar{c}_{i,k'}) = \bar{c}_{i,k'}$, and $\pi^*(c_{i,k}) = c_k$.
\end{lem}

\begin{proof}
	
The surface $X$ is homotopy equivalent to a wedge of $2 \hat{g} +n$ circles and the this equivalence send the boundary components $S_i^1$ for $i=1,...,n$ to circles in the wedge product. The lemma now follows directly from Lemma \ref{oaarceublrc}.	
\end{proof}

Using the coordinates of Lemma \ref{aolercubalcroe}, the Koszul resolution of $R \rightarrow H^*(BMap_0(X,U_r))$ is the differential bigraded algebra $(K^{*,*},\delta)$, where
$$ K^{*,*} := \wedge(\bar{c}_{i,k'}, x_{i',k} ) \otimes S(c_{i,k}) \otimes A $$
with bidegrees and differentials

\begin{tabular}{|c|c|c|}
	\hline
generator   & bi-degree  & $\delta$-derivative \\
\hline
$\bar{c}_{i,k'}$ & $(0,2k'-1)$ & $0$ \\
$c_{i,k}$ & $(0,2k)$ & $0$\\
$x_{i',k}$ & $(-1,2k)$ & $c_{i',k}+c_{1,k}$ \\
\hline	
\end{tabular}

Note in particular that $K^{*,*}$is a free extension over $R$ and the cohomology $H(K^{*,*},\delta)$ is isomorphic to $H^*(BMap_0(X,U_r))$ as an $R$-module, where we understand elements in $H^d(BMap_0(X,U_r))$ to have bi-degree $(0,d)$. By (\ref{Oeularocebr}), $EM_2^{*,*}$ is isomorphic as a bi-graded algebra to the homology of the complex 
$$ ( K^{*,*} \otimes_R H^{*}(\prod_{i=1}^n BLU_r^{\tau_i}),~ \delta \otimes_R 1). $$
Applying Corollary \ref{orecubrobeui} and the Kunneth theorem, we have an isomorphism of $R$-modules
	$$ H^*(\prod_{i=1}^n BLU_r^{\tau_i}) \cong V \otimes S(c_{i,k})$$  	
where $V$ is a graded vector space with Poincar\'e series
$$P_t(V) = \frac{1}{(1+t^r)^{a}}\prod_{k=1}^{r}(1+t^k)^{2a}(1+t^{2k-1})^{n-a}, $$ and the $R$-module structure is defined by $R \rightarrow V \otimes S(c_{i,k})$, $c_{i,k} \mapsto c_{i,k}$ and $\bar{c}_{i,k'} \mapsto 0$.

Forming the tensor product gives
$$K^{*,*} \otimes_R H^{*}(\prod_{i=1}^n BLU_r^{\tau_i}) \cong V \otimes A \otimes \wedge( x_{i',k} ) \otimes S(c_{i,k}) $$
This complex factors into $V \otimes A$ with trivial differential and the Kozsul complex $\wedge( x_{i',k} ) \otimes S(c_{i,k})$ with differential $\delta(x_{i',k} ) = c_{i',k} + c_{1,k}$ whose homology is simply $S(c_k)$. Applying the Kunneth theorem for chain complexes gives
$$ EM_2 = V \otimes A  \otimes S(c_k). $$
This bigraded algebra is zero outside of the column $EM_2^{0,*}$, so it must collapse and we deduce $$P_t(B \tilde{\G}(\hat{g}, n, r; \tau_1,...,\tau_n)) = P_t(V) P_t(A) P_t(S(c_k))$$ completing the proof.
\end{proof}

\begin{remark}\label{etrgfffre}
In the proof of Lemma \ref{oeracboel}, we showed that $EM_{\infty}$ is supported in the zeroth column. It follows from  Lemma \ref{ezero} that the induced map $H^*(\prod_{i=1}^n BLU_r^{\tau_i}) \otimes H^*(BMaps_0(X, U_r)) \rightarrow H^*(B \tilde{\G}(\hat{g}, n, r; \tau_1,...,\tau_n))$ is injective. \end{remark}

\subsection{The second spectral sequence}\label{secondspec}

The group $\G(\hat{g},n, r; \tau_1,...,\tau_n)$ may be identified with the subgroup of $\tilde{\G}(\hat{g}, n, r; \tau_1,...,\tau_n) \subset Maps_0(X, U_r)$ consisting of those elements that take constant value on the remaining boundary circle $S\subseteq \partial X$. This determines a pull-back diagram of topological groups,

\begin{equation*}
\xymatrix{ \G(\hat{g},n, r; \tau_1,...,\tau_n) \ar[r] \ar[d] & \tilde{\G}(\hat{g}, n, r; \tau_1,...,\tau_n) \ar[d]^{\pi}\\
U_r \ar[r] &  L_0U_r }
\end{equation*}
where $\pi$ is restriction to the boundary circle $S$. Applying the classifying space functor produces a fibre bundle pull-back
\begin{equation}\label{arcoelrb}
\xymatrix{ B\G(\hat{g},n, r; \tau_1,...,\tau_n) \ar[r] \ar[d] & B\tilde{\G}(\hat{g}, n, r; \tau_1,...,\tau_n) \ar[d]^{\pi}\\
BU_r \ar[r] &  BL_0U_r }
\end{equation}

\begin{lem}\label{aloerublaroceb}
The Eilenberg-Moore spectral sequence of the diagram (\ref{arcoelrb}) has second page the bigraded algebra
$$ EM_2^{*,*} \cong  \Gamma(z_2,...,z_r) \otimes H^*(B\tilde{\G}(\hat{g}, n, r; \tau_1,...,\tau_n)) $$
where $z_{k'}$ has bi-degree (-1,2k'-1), $\Gamma(z_2,...,z_r)$ denotes the divide power algebra on generators $z_2,...,z_r$ and $H^d(B\tilde{\G}(\hat{g}, n, r; \tau_1,...,\tau_n))$ is given bidegree $(0,d)$ (i.e. lies in the 0th column).
\end{lem}

\begin{proof}
By Lemma \ref{oaarceublrc} we have isomorphisms $H^*(BL_0U_r) \cong \wedge(\bar{c}_2,...\bar{c}_r) \otimes S(c_1,...,c_r)$ and $H^*(BU_r) \cong S(c_1,...,c_r)$.  The morphism $H^*(BL_0U_r) \rightarrow H^*(BU_r)$ sends $c_k$ to $c_k$ and $\bar{c}_{k'}$ to $0$.  The associated Koszul resolution $(K^{*,*}, \delta)$ is
$$ K^{*,*} \cong \Gamma(z_2,...,z_r) \otimes \wedge(\bar{c}_2,...,\bar{c}_r) \otimes S(c_1,...,c_r) = \Gamma(z_2,...,z_r) \otimes H^*(BL_0U_r)$$
with generators satisfying

\begin{tabular}{|c|c|c|}
	\hline
generator   & bi-degree  & $\delta$-derivative \\
\hline
$\bar{c}_{k'}$ & $(0,2k'-1)$ & $0$ \\
$c_{k}$ & $(0,2k)$ & $0$\\
$z_{k'}$ & $(-1,2k-1)$ & $\bar{c}_{k'}$ \\
\hline	
\end{tabular}

The morphism $\pi^*: H^*(BL_0U_r) \rightarrow H^*(B\tilde{\G}(\hat{g}, n, r; \tau_1,...,\tau_n))$ sends $\bar{c}_{k'}$ to $0$ for all $k' = 2,...,r$, so the tensor product complex $ K^{*,*} \otimes_{H^*(BL_0U_r)} H^*(B\tilde{\G}(\hat{g}, n, r; \tau_1,...,\tau_n)) $ has trivial boundary operator.  We conclude that
$$ EM_2^{*,*} = K^{*,*} \otimes_{H^*(BL_0U_r)} H^*(B\tilde{\G}(\hat{g}, n, r; \tau_1,...,\tau_n))  = \Gamma(z_2,...,z_r) \otimes H^*(B\tilde{\G}(\hat{g}, n, r; \tau_1,...,\tau_n)). $$
\end{proof}

To complete the proof of Theorem \ref{alcgjbrcd.}, it remains to prove that the spectral sequence of Lemma \ref{aloerublaroceb} collapses at $EM_2$. We turn to this tricky problem in \S  \ref{collapse}.

\subsection{Collapsing the spectral sequence}\label{collapse}

The first idea is to stabilize with respect to rank. The \emph{trivial real line bundle} over a real space $(M,\sigma)$ is the line bundle $M\times \C$ with involution $\tau_{triv}(m,z) = (\sigma(m), \bar{z})$.

\begin{lem}\label{aoeublracoe}
The morphism of pull-back diagrams (\ref{arcoelrb}) induced by forming a direct sum with the trivial real line bundle $$B\G(\hat{g},n, r; \tau_1,...,\tau_n) \rightarrow B\G(\hat{g},n, r+1; \tau_1\oplus \tau_{triv},...,\tau_n\oplus \tau_{triv})$$ 
determines a surjection on $EM_2$.
\end{lem}

\begin{proof}
This is a routine check using functoriality of diagrams (\ref{alorecudcr}) and Lemma \ref{aloerublaroceb}. 
\end{proof}

An immediate consequence of Lemma \ref{aoeublracoe} is that the EMSS for $B\G(\hat{g},n, r; \tau_1,...,\tau_n) $ collapses if the EMSS of  $B\G(\hat{g},n, r+1; \tau_1\oplus \tau_{triv},...,\tau_n\oplus \tau_{triv})$ collapses.  In particular, we may focus on direct limit 
$$B\G(\hat{g},n; \tau_1,...,\tau_n) := \lim_{s\rightarrow \infty} B\G(\hat{g},n, r+s; \tau_1 \oplus \tau_{triv}^s,...,\tau_n \oplus \tau_{triv}^{s}).$$

By working in the stable limit, we gain the following simplification.

\begin{lem}\label{aolercuba}
The homotopy type of $B\G(\hat{g},n; \tau_1,...,\tau_n)$ is independent of the degree and Stieffel-Whitney numbers of the associated real vector bundle. 
\end{lem}

\begin{proof}

First recall that $BU$ is an H-space under the map $m: BU \times BU \rightarrow BU$ defined as the direct limit of the maps $BU_r \times BU_r \stackrel{\oplus}{\rightarrow} BU_{2r}$. The multiplication map $m$ clearly commutes with complex conjugation action on $BU$ so for any $\Z/2$-space $Y$ the space of equivariant maps of the form $Maps_{\Z_2}(Y, BU)$ becomes an H-space by point-wise multiplication.

Applying the classifying space functor to stable version of diagram (\ref{alorecbulrcoeb}), we obtain

\begin{equation}
\xymatrix{ B\G(\hat{g},n; \tau_1,...,\tau_n) \ar[r] \ar[d] & BMaps(\Sigma(\hat{g},n), U) \ar[d]^{\pi}\\
\prod_{i=1}^n BLU^{\tau_i} \ar[r] &  \prod_{i=1}^n BLU  }
\end{equation}

Applying Proposition \ref{lracldio} we find that $B\G(\hat{g},n, r; \tau_1,...,\tau_n)$ is identified with a path component of the space $H$ defined by the homotopy pull-back diagram of H-spaces

\begin{equation}
\xymatrix{ H \ar[r] \ar[d] & Maps(\Sigma(\hat{g},n), BU) \ar[d]^{\pi}\\
\prod_{i=1}^n (LBU)^{\sigma_i} \ar[r] &  \prod_{i=1}^n LBU  }
\end{equation}
where $(LBU)^{\sigma_i} = Maps^{\Z/2}( (S^1,\sigma_i), (BU, \bar{\cdot}))$. Because $H$ is an H-space for which $\pi_0(H) \cong \pi_0(\prod_{i=1}^n (LBU)^{\sigma_i}  ) \cong (\Z/2)^a$ is a group (here $a$ is the number of path components of $\Sigma^{\sigma}$), it follows that the path components of $H$ are pair-wise homotopy equivalent.
\end{proof}

Consider now the stable version of (\ref{arcoelrb})
\begin{equation}\label{aolrecubalrcoi}
\xymatrix{ B\G(\hat{g},n; \tau_1,...,\tau_n) \ar[r] \ar[d] & B\tilde{\G}(\hat{g}, n; \tau_1,...,\tau_n) \ar[d]^{\pi}\\
BU\ar[r] &  BL_0U }
\end{equation}
where we set all Stieffel-Whitney classes to zero. We are reduced to showing that the EMSS associated to \ref{aolrecubalrcoi} collapses.

\begin{lem}
The EMSS associated to (\ref{aolrecubalrcoi}) collapses at $EM_2$ if and only if the morphism
\begin{equation} \label{laoelrucabo}
H^*(B\tilde{\G}(\hat{g}, n; \tau_1,...,\tau_n) ) \rightarrow H^*(B\G(\hat{g},n; \tau_1,...,\tau_n))
\end{equation} is injective.
\end{lem}

\begin{proof}

By the stable version of Lemma \ref{aloerublaroceb}, we have an isomorphism of bigraded algebras 

\begin{equation}\label{aoleublrc}
EM_2^{*,*} \cong \Gamma(z_2,z_3,...) \otimes H^*(B \tilde{\G}(\hat{g}, n; \tau_1,...,\tau_n)).
\end{equation}
In spectral sequence terms, we want to show that $EM_2= EM_{\infty}$ if and only if the column $EM_2^{0,*} =  1 \otimes H^*(B \tilde{\G}(\hat{g}, n; \tau_1,...,\tau_n)$ survives to infinity.  The ``only if" direction is clear.

Arguing in the same fashion as the proof of Lemma \ref{aolercuba}, we find that (\ref{aolrecubalrcoi}) is a pull-back diagram of H-spaces.  By (Smith \cite{smith1970lectures} chap. 2) $EM_*^{*,*}$ is a spectral sequence of (connected, commutative and cocommutative) Hopf algebras (we refer to  Milnor-Moore \cite{milnor1965structure} for background on Hopf algebras).

Suppose now that $EM_*$ does not collapse at $EM_2$.  Then for some  $r\geq 2$,  $EM_2^{*,*} = EM_r^{*,*}$ and the coboundary map $d_r$ is non-trivial.  According to Lemma \ref{rocrdlroced;j}, there must exist an indecomposable element $q \in EM_r$ and a non-zero primitive element $p \in P(EM_r)$ such that $d_r(q) = p$. By (\ref{aoleublrc}), all odd total degree indecomposables lie in the zeroth column and thus must be $d_r$-closed. It follows that $q$ must have even total degree and $p$ has odd total degree.  On the other hand, by \cite{milnor1965structure} (Prop. 4.21) decomposable primitives must lie in the image of the Frobenius morphism, hence have even degree.  Thus all odd degree primitives must be indecomposable, so $p$ must lie in the zeroth column $EM_2^{0,*}$. We deduce that  (\ref{laoelrucabo}) is not injective unless $EM_{\infty}^{*,*} = EM_2^{*,*}$.
\end{proof}

We are reduced to proving that (\ref{laoelrucabo}) is injective.  We begin with the genus zero case. Our strategy is to reverse the usual Atiyah-Bott argument by computing the Betti numbers of the real moduli space directly, and then using the recursive formula to compute $P_t(B\G^{\tau}_E)$.

Let $\M_{(\Sigma, \sigma)}(r,d,\tau) = C_{ss}(r,d,\tau)_{h\G_{\C}(r,d,\tau)}$ denote the topological moduli stack of rank $r$, degree $d$ real bundles of type $\tau$.
We consider two involutions $\sigma_a, \sigma_b: \C P^1 \rightarrow \C P^1$ where $\sigma_a$ fixes a circle and $\sigma_b$ has no fixed points (for example, in homogeneous coordinates $\sigma_a([z_1: z_2]) = [\bar{z}_1, \bar{z}_2]$ and $\sigma_b([ z_1, z_2 ]) = [-\bar{z}_2, \bar{z}_1]$). 
\begin{prop}
The moduli stacks satisfy homotopy equivalences
$$ \M_{(\C P^1, \sigma_a)}(r,d,\tau) \cong \M_{(\C P^1, \sigma_b)}(r,0,\tau)  \cong BO_r $$
$$ \M_{(\C P^1, \sigma_b)}(2r,2r,\tau) \cong  \M_{(\C P^1, \sigma_b)}(2r,-2r,\tau)  \cong BSp_{2r}. $$
\end{prop}

\begin{proof}
Let $\E \rightarrow \C P^1$, be a semistable holomorphic bundle.  Then by \S \ref{cp1complex} we know $\E\cong \mathcal{O}(k)^{\oplus r}$ for some $k = \deg(\E)/r$, and $Aut(E) \cong GL_r(\C)$.  Combined with the topological classification of real bundles (Theorem \ref{arecbu}), we find that up to isomorphism there is at most one semistable real bundle of given rank and degree over $\C P^1$. It follows that 
$$\M_{(\C P^1, \sigma)}(r,kr,\tau) \cong BAut(r,kr,\tau) $$ 
where $Aut(r,kr,\tau) \subseteq Aut(\mathcal{O}(k)^{\oplus r}) \cong GL_r(\C)$ is the subgroup that commutes with the real involution. 

In the $\sigma_a$ case choose $p \in (\C P^1)^{\sigma_a}$.  Then we may model $\mathcal{O}(k) = \mathcal{O}(kp)$ as the sheaf of meromorphic functions with poles of order at most $k$ at $p$, with $\tau$ acting in the obvious way. The real subgroup $Aut(r,kr,\tau) \subseteq GL_r(\C)$ in this case is easily identified with $GL_r(\R)$.  

In the $\sigma_b$ and $k=0$ case, we have $\E = \C P^1 \times \C^r$ trivial and the isomorphism $GL_r(\C) = Aut(\E)$ can be understood acting in the standard way on the $\C^r$ factor and we have $Aut(\E,\tau) \cong GL_r(\R)$. In the case $\M_{(\C P^1, \sigma_b)}(2r,\pm 2r,\tau)$, tensoring by a degree $\pm 1$ quaternionic line bundle produces an isomorphism with the moduli space of rank $2r$ and degree $0$ quaternionic bundles on $\C P^1$, which by similar reasoning  has automorphism group $Sp_{r}(\C) \subseteq GL_{2r}(\C)$.  	
\end{proof}

\begin{lem}
Over a genus zero curve, the Poincar\'e polynomial of the classifying spaces of stable real gauge groups satisfy
$$ P_t(B\G(0,1;\tau_a)) =   \prod_{k=1}^{\infty}\frac{1}{ (1-t^{k})^2}.$$
$$ P_t(B\G(0,1;\tau_b)) = \prod_{k=1}^{\infty}   \frac{1+t^{2k-1}}{(1-t^{2k})^2}.$$
Consequently, the EMSS of Lemma \ref{aloerublaroceb} collapses in the genus zero case.
\end{lem}

\begin{proof}
As explained in  \S \ref{cp1complex}, in the stable limit $r \rightarrow \infty$ the only contributions to the recursive formula
are Harder-Narasimhan strata of the form $((n,n), (r-2n,0), (n,-n))$. In the $\tau_a$ case, the recursive formula (\ref{aolrcubalocre}) gives,
\begin{eqnarray*} P_t(B\G(0,1;\tau_a)) 
	&=& \sum_{n=0}^{\infty} t^{n^2} P_t(BO_n)^2 P_t( \lim_{r\rightarrow \infty} BO_{r-2n})\\
	 &= & P_t(BO) \sum_{n=0}^{\infty} t^{n^2}P_t(BO_n)^2 \\
	 &=&    \Big( \prod_{k=1}^{\infty} \frac{1}{1-t^{k}}\Big)\sum_{n=0}^{\infty}  t^{n^2}\prod_{k=1}^n \frac{1}{(1-t^{k})^2}\\
	 & = & \prod_{k=1}^{\infty}\frac{1}{ (1-t^{k})^2}.  \end{eqnarray*}
	 where the last equality is deduced from (\ref{darinking}) by replacing $t^2$ by $t$.  
For the $\tau_b$ case the formula (\ref{aolrcubalocre}) is altered by the fact that real bundles only exist in even degree and consequently only HN-strata of the form $((2n,2n), (2r-4n,0), (2n,-2n)) $ contribute. In this case, (\ref{aolrcubalocre}) gives
\begin{eqnarray*} P_t(B\G(0,1;\tau_b)) 
	 &= &P_t(BO)\sum_{n=0}^{\infty} t^{4n^2} P_t(BSp_n)^2  \\
	& =&    \Big( \prod_{k=1}^{\infty} \frac{1}{1-t^k}\Big)\Big( \sum_{n=0}^{\infty}  t^{4n^2}  \prod_{k=1}^n\frac{1}{ (1-t^{4k})^2} \Big) \\
& = &\Big( \prod_{k=1}^{\infty} \frac{1+t^k}{1-t^{2k}}\Big) \Big( \prod_{k=1}^{\infty}\frac{1}{ 1-t^{4k}} \Big)
\\& =&  \prod_{k=1}^{\infty} \frac{1+t^k}{(1-t^{2k})^2(1+t^{2k})}\\
& =&  \prod_{k=1}^{\infty}   \frac{1+t^{2k-1}}{(1-t^{2k})^2}
\end{eqnarray*}
where we have employed the identity (\ref{oeulracbo}) with $x = t^4$. 
\end{proof}

Consider now the wedge product of surfaces $ Y := \Sigma(\hat{g}, 0 ) \vee (  \vee_n  \Sigma(0,1)) $ where we choose base points not lying on boundaries. Here $\Sigma(\hat{g},0)$ is the closed surface of genus $\hat{g}$ and $\Sigma(0,1)$ is a disk. Because $Y$ has $n$ boundary circles coming from the $n$ copies of $\Sigma(0,1)$, we can define by analogy with (\ref{alorecbulrcoeb}) the group $\G_Y^{\tau}$ via the pull-back diagram

$$\xymatrix{\G_Y^{\tau} \ar[d] \ar[r]&  Maps_0(Y, U_r) \ar[d]\\ \prod_{i=1}^n LU_r^{\tau_i} \ar[r] & \prod_{i=1}^n LU_r  } .$$

We fit $Y$ into a commutative diagram of spaces
$$ \xymatrix{  Y = \Sigma(\hat{g}, 0 ) \vee (  \vee_n  \Sigma(0,1))  & \Sigma(\hat{g}, n) \ar[l] \\    \Sigma(\hat{g}, 0) \sqcup ( \sqcup_{n} \Sigma(0,1) )  \ar[u]& X \sqcup (\sqcup_n S^1) \ar[l]  \ar[u] }.$$
where as before $X$ is the surface $\Sigma(\hat{g}, n)$ with a disk removed. These maps of surfaces induce homomorphisms of gauge groups and ultimately a commuting diagram
$$ \xymatrix{   H^*(B\G_Y^{\tau})  & H^*(B\G(\hat{g}, n; \tau_1,...,\tau_n)) \ar[l] \\    H^*(B\G(\hat{g}, 0)) \otimes H^*( \prod_{i=1}^n B\G(0,1; \tau_i) ))  \ar[u]^{\phi_2}& H^*(BMaps_0(X,U)) \otimes H^*(\prod_{i=1}^n BLU^{\tau_i})) \ar[l]^{\phi_1}  \ar[u]^{f}  }$$
By Remark \ref{etrgfffre}, the image of $f$ coincides with the image of $(\ref{laoelrucabo})$. Thus, to prove that $(\ref{laoelrucabo})$ is injective, it suffices to prove the following lemma.

\begin{lem} The Poincar\'e series of the image of $\phi_2 \circ \phi_1$ is equal to the Poincar\'e series $P_t(B\tilde{\G}(\hat{g},n;\tau_1,...,\tau_n))$.
\end{lem}

\begin{proof}
From Lemma \ref{aolercubalcroe} and Corollary \ref{orecubrobeui} , we know that 
$$P_t( BMaps_0(X,U) \times \prod_{i=1}^n BLU^{\tau_i}) = (1+t)^{-n}\prod_{k=1}^{\infty}\frac{(1+t^{2k-1})^{2\hat{g}+2n-a}(1+t^k)^{2a}}{(1-t^{2k})^{n+1}}  $$
The first morphism $\phi_1$ is the tensor product of the injections $H^*(BLU^{\tau_i}) \rightarrow H^*(B\G(0,1;\tau_i))$ and the map $H^*(BMaps_0(X, U)) \rightarrow H^*(B\G(\hat{g},0))$ induced by the inclusion of the punctured surface $X$ into the genus $g$ surface $\Sigma(\hat{g},0)$. This kills only the cohomology coming from the boundary loops (see Lemma \ref{oaarceublrc})
and we deduce that the image of $\phi_1$ has Poincar\'e series $$ P_t(Im(\phi_1)) = \prod_{k=1}^{\infty}\frac{(1+t^{2k-1})^{2\hat{g}+n-a}(1+t^k)^{2a}}{(1-t^{2k})^{n+1}} $$
Next the kernel of $\phi_2$ is generated as an ideal by the classes $c_{k} - c_{k,i}$ for $k=1,...,\infty$ and $i=1,...,n$. All of these classes lie in the image of $\phi_1$, so  $Im( \phi_2 \circ \phi_1)$ has Poincar\'e series
\begin{eqnarray*}P_t(Im( \phi_2 \circ \phi_1)) & = &P_t(Im(\phi_1)) \prod_{i=1}^n \prod_{k=1}^{\infty}(1-t^{2k}) \\ &=& \prod_{k=1}^{\infty}\frac{(1+t^{2k-1})^{2\hat{g}+n-a}(1+t^k)^{2a}}{(1-t^{2k})}
\end{eqnarray*}
which equals  $P_t(B\tilde{\G}(\hat{g},n;\tau_1,...,\tau_n))$ by Lemma \ref{oeracboel}.
\end{proof}

\section{Betti numbers of moduli spaces}\label{sect7}
Let $(E,\tau)\rightarrow (\Sigma,\sigma)$ be a $C^{\infty}$-real bundle and consider the short exact sequence
\begin{equation}\label{onetbunoeb}
 1\rightarrow C_2 \rightarrow \G^{\tau}_E \rightarrow \bar{\G}^{\tau}_E\rightarrow 1
\end{equation}
where $C_2$ is the subgroup of constant maps with value $\pm Id_{U_r}$.

\begin{lem}\label{osrecbuarsc}
If either	
\begin{itemize}
	\item the rank $r$ of $E$ is odd or,
	\item  $w_1(E^{\tau}) \neq 0$ in $H^1(\Sigma^{\sigma};\Z/2)$.
\end{itemize}
then (\ref{onetbunoeb}) splits to define an isomorphism $\G^{\tau}_E \cong C_2 \times \bar{\G}^{\tau}_E$.
In particular, if $\G^{\tau}_E$ acts on a finite type space $X$ such that $C_2$ acts trivially, then   
$$P_t^{\overline{\G}^{\tau}_E}(X) = (1-t) P_t^{\G^{\tau}_E}(X).$$	
\end{lem}	
	
\begin{proof}

Because $C_2 \subset \G^{\tau}_E$ is central, it suffices to prove that there is some homomorphism $\phi: \G^{\tau}_E \rightarrow \Z/2$ mapping $C_2$ isomorphically onto $\Z/2$. If $r$ is odd, then this can be accomplished simply by taking the determinant of the gauge group action at a fibre.  

It remains to consider the even rank case $r =2n$ and non-trivial $w_1(E^{\tau})$. Necessarily, $\Sigma^{\sigma}$ is non-empty. By factoring through the restriction to an invariant circle $\G^{\tau}_E \rightarrow LU_r^{\tau}$ we only need a homomorphism $LU_r^{\tau} \rightarrow \Z/2$ separating the constant loop $-1$ from the identity. In this case, we can use the model 
$$LU_r^{\tau} \cong L_gO_r = \{ \gamma:I \rightarrow O_r| \gamma(0) = g\gamma(2\pi)g^{-1}\}$$ where $g \in O_r$ has determinant $-1$.  This model determines a short exact sequence of groups
$$1\rightarrow \Omega SO_r \stackrel{i}{\rightarrow} LU_r^{\tau} \stackrel{\rho}{\rightarrow} O_r \rightarrow 1 $$
where $\rho(\gamma) = \gamma(0)$ and an exact sequence on $\pi_0$
\begin{equation}\label{Aoeulcobec}
 \pi_0(\Omega SO_r) \stackrel{i_*}{\rightarrow} \pi_0(LU_r^{\tau}) \rightarrow \pi_0(O_r)  
\end{equation}
where $\pi_0(\Omega O_r)$ and $\pi_0(O_r)$ are cyclic groups of order $2$. It follows that $\pi_0(LU_r^{\tau})$ has order at most four. On the other hand we have natural isomorphisms 

\begin{eqnarray*} Hom( \pi_0(LU_r^{\tau}),\Z/2) & =&  Hom( \pi_1(BLU_r^{\tau}),\Z/2)\\ & =& Hom(H_1(BLU_r^{\tau});\Z), \Z/2)\\& =& H^1(BLU_r^{\tau}, \Z/2)  \cong  (\Z/2)^2
\end{eqnarray*}
where the last isomorphism follows from Proposition \ref{robelurcbe}. We conclude that $\pi_0(LU_r^{\tau}) \cong (\Z/2)^2$, so it is enough to show that the constant loop $-1 \in L_gO_r$ does not lie in identity path component. By a homotopy extension argument, the $-1$ is homotopic to the concatenation $\gamma \cdot (g\gamma g^{-1})$ where $\gamma : I \rightarrow SO_n$ is any path in $SO_r$ with $\gamma(0) = 1$ and $\gamma(1) = -1$. But $\gamma \cdot (g\gamma g^{-1})$ represents the generator of $\pi_1(SO_r) = \pi_0(\Omega SO_r) =\Z/2$. Finally $i^*$ of (\ref{Aoeulcobec}) is injective, so $-1 \in LU_r^{\tau}$ does not lie in the identity component.

Finally, if (\ref{onetbunoeb}) and $\G_E^{\tau}$ acts on $X$ with $C_2$ acting trivially, then $ X_{h\G_E^{\tau}} = BC_2 \times X_{h\bar{\G}_E^{\tau}}$ and the identity of Poincar\'e series follows.
\end{proof}

We are now able to compute some Poincar\'e polynomials. To begin with a simple example, consider the case of rank $r=1$. In this case, all real bundles are semistable, so 

\begin{equation}\label{aoerculrci}
 P_t(M(1,d,\tau)) = (1-t) P_t(C_{ss}(1,d,\tau) )= (1-t) P_t(B\G(1,d,\tau) ) = (1+t)^g
\end{equation}
where in the last step we employ the formula $P_t(B\G(1,d,\tau)) = \frac{(1+t)^g}{1-t}$ . Of course, it is known since Gross-Harris \cite{gross1981real} that $M(1,d,\tau)$ is homeomorphic to $(S^1)^g$, so (\ref{aoerculrci}) is not new. Next, we consider rank two.

\begin{prop}
	Let $\Sigma$ be a genus $g$ real curve with $a >0$ real path components and set $b := a-1$. The moduli space $M(2,d, \tau)$ of real bundles of rank two, odd degree $d$ and fixed topological type has Poincar\'e series

\begin{equation}\label{ercaboecub}
P_t(M(2,d, \tau))= \frac{(1+t)^{g+b}(1+t^2)^{b}(1+t^3)^{g-b} - 2^bt^{g}(1+t)^{2g}}{(1-t)(1-t^2)}.
\end{equation}
\end{prop}

\begin{proof}
For simplicity, we set $d=1$. The remaining odd degrees cases are isomorphic by tensoring with a real line bundle.	
	
Because the rank and degree are coprime, the action of $\G(2,1,\tau)$ on $C_{ss}(2,1,\tau)$ has constant stabilizer $\Z/2$. Thus, according to Lemma \ref{osrecbuarsc}.
$$ P_t(M(2,1,\tau)) = P_t^{\bar{\G}(2,1,\tau)} (C_{ss}(2,1,\tau))  = (1-t)P_t^{\G(2,1,\tau)} (C_{ss}(2,1,\tau))$$

We wish to apply the recursive formula (\ref{aolrcubalocre}). Complex Harder-Narasimhan types are determined by a splitting $E = L_1 \oplus L_2$ into line bundles with $\deg(L_1) > \deg(L_2)$. For each such complex splitting of $E$, there are $2^{a-1} = 2^b$ real Harder-Narasimhan types determined by possible choices of Stieffel-Whitney numbers, and each higher stratum has Poincar\'e series $\Big(\frac{(1+t)^g}{1-t}\Big)^2$. The recursive formula becomes 
	
\begin{eqnarray*} P_t^{\G(2,1,\tau)}(C_{ss}(2,1,\tau))  &= & P_t(B\G(2,1,\tau)) - \sum_{i=1}^{\infty} t^{2i-1 +(g-1)} \Big(\frac{(1+t)^g}{1-t}\Big)^2\\
	&=& \frac{(1+t)^{g+b}(1+t^2)^b(1+t^3)^{g-b}}{(1-t)^2(1-t^2)} - \frac{2^b t^g(1+t)^{2g}}{(1-t)^2(1-t^2)} .
\end{eqnarray*}

\end{proof}

\begin{remark}
	If $(\Sigma, \tau)$ be a real curve of genus $g$, with $g+1$ real path-components, then (\ref{ercaboecub}) proves a conjectural formula due to Saveliev-Wang \cite{saveliev2010real}.
\end{remark}

For example, for a real curve of genus $g=2$, respectively $a= 1,2,3$, $P_t(M(2,1,\tau))$ equals
$$t^5 + 3t^4 + 4t^3 + 4t^2 + 3t + 1 $$
$$t^5 + 4t^4 + 7t^3 + 7t^2 + 4t + 1 $$
$$t^5 + 5t^4 + 10t^3 + 10t^2 + 5t + 1 $$

For a real curve of genus $g=3$, $a=1,2,3,4$, $P_t(M(2,1,\tau))$ equals
$$t^9 + 4t^8 + 8t^7 + 14t^6 + 21t^5 + 21t^4 + 14t^3 + 8t^2 + 4t + 1 $$
$$t^9 + 5t^8 + 13t^7 + 25t^6 + 36t^5 + 36t^4 + 25t^3 + 13t^2 + 5t + 1 $$
$$t^9 + 6t^8 + 19t^7 + 41t^6 + 61t^5 + 61t^4 + 41t^3 + 19t^2 + 6t + 1 $$
$$t^9 + 7t^8 + 26t^7 + 62t^6 + 96t^5 + 96t^4 + 62t^3 + 26t^2 + 7t + 1 $$

For rank $r$ greater than $2$, the calculation of $P_t(M(r,d,\tau))$ using recursion involves multiple iterated geometric series.

\begin{prop}
Let $\Sigma$ be a genus $g$ real curve with $a >0$ real path components and set $b := a-1$ and let $d$ be an integer relatively prime to $3$. The moduli space $M(3,d, \tau)$ of real bundles of rank three, degree $d$ and fixed topological type has Poincar\'e series
\begin{eqnarray*} P_t(M(3,d, \tau)) &=&  \frac{(1+t)^{g+b}(1+t^2)^{2b}(1+t^3)^g(1+t^5)^{g-b}}{(1-t)(1-t^2)^2(1-t^3)}\\ 
&&- 2^b\frac{t^{2g}(1+t)^{2g+b}(1+t^2)^b(1+t^3)^{g-b}}{t(1-t)^3(1-t^3)}	\\   
&&+ 4^b\frac{t^{3g}(1+t)^{3g}(1+t^2+t^4) }{t(1-t)^2(1-t^2)(1-t^6)}. \end{eqnarray*}
\end{prop}

\begin{proof}
This is a combinatorial exercise.
\end{proof}

\begin{remark}
A combination of tensoring by real line bundles or dualizing produces a homeomorphism between any two real moduli spaces $M(3,d,\tau)$ and $M(3, d',\tau')$ for which $d$ and $d'$ relatively prime to $3$. This explains why the above formula is independent of degree and Stieffel-Whitney numbers.
\end{remark}

For example, for genus $g=2$ and $a=1,2,3$, $P_t(M(3,1,\tau))$ equals
$$t^{10} + 3t^9 + 6t^8 + 12t^7 + 17t^6 + 18t^5 + 17t^4 + 12t^3 + 6t^2 + 3t
    + 1 $$
$$	t^{10} + 4t^9 + 11t^8 + 25t^7 + 40t^6 + 46t^5 + 40t^4 + 25t^3 + 11t^2 +
	    4t + 1 $$ 
$$	t^{10} + 5t^9 + 17t^8 + 44t^7 + 78t^6 + 94t^5 + 78t^4 + 44t^3 + 17t^2 +
		    5t + 1 $$

\begin{remark}Liu and Schaffhauser (\cite{liu2011yang} section 6.2) have produced a closed formula for $P_t(M(r,d,\tau))$ for all $r$, $d$ and $\tau$ by solving the recursion relation.
\end{remark}

\begin{appendix}

\section{Review of the Eilenberg-Moore spectral sequence}\label{EM spec seq}\label{aolrecubque}

We summarize the relevant parts of section 7.1 of McLeary \cite{mccleary2001user}. Let $F \rightarrow E \stackrel{\pi}{\rightarrow} B$ be a fibre bundle with $F$ connected and $B$ is simply connected. Given a continuous map $f: X \rightarrow B$ we may form the pull-back fibre bundle

\begin{equation}\label{rloceuclrb}\begin{CD}
	\xymatrix{ E_f \ar[r] \ar[d] & E \ar[d]^{\pi}\\
	            X \ar[r]^f & B }
\end{CD}\end{equation}

The Eilenberg-Moore spectral sequence is a second quadrant spectral sequence of bigraded algebras $(EM_r^{p,q},\delta_r)$ converging strongly to an associated graded of $H^*(E_f)$ for which $$ E_2^{*,*} = Tor^{*,*}_{H^*(B)}(H^*(X),H^*(E)) $$ 
where $H^*(X)$ and $H^*(E)$ are $H^*(B)$-modules via $f^*$ and $\pi^*$. The boundary maps are bi-graded $\delta_r: EM_r^{p,q} \rightarrow EM_r^{p+r, q-r+1}$. 
\begin{lem}[\cite{mccleary2001user} Proposition 8.23]\label{ezero}
For the EMSS associated to the pull-back digram (\ref{rloceuclrb}), the column $EM_{\infty}^{0,*}$ may be identified with subalgebra of $H^*(E_f)$ generated by $\im(\pi^*)$ and $\im(f^*)$. 
\end{lem}
The EMSS is functorial with respect to morphisms of diagrams

\begin{equation}\label{alorecudcr} \xymatrix{ X \ar[r]  \ar[d]^{\phi} & B \ar[d]^{\phi}  & \ar[l] E \ar[d]^{\phi}  \\ X' \ar[r] & B' & \ar[l] E'  }\end{equation}
and the map on $EM_2$ is the standard algebraic map $$  Tor^{*,*}_{H^*(B')}(H^*(X'),H^*(E'))  \rightarrow Tor^{*,*}_{H^*(B)}(H^*(X),H^*(E)) $$ induced by the homomorphisms of cohomology rings $\phi^*$.

In case (\ref{rloceuclrb}) is a diagram of H-spaces, $EM_*^{*,*}$ becomes a spectral sequence of Hopf algebras as explained in Smith \cite{smith1970lectures} chapter 2. 

\begin{lem}[\cite{mccleary2001user} Lemma 7.13]\label{rocrdlroced;j}
If $(E_r, d_r)$ is a spectral sequence of Hopf algebras, then for each $r$, in the lowest degree that $d_r$ is non-trivial, it is defined on an indecomposable element and has as value a primitive element. 
\end{lem}

\end{appendix}

\bibliographystyle{amsalpha}
\bibliography{TomReferences}

\providecommand{\bysame}{\leavevmode\hbox to3em{\hrulefill}\thinspace}
\providecommand{\MR}{\relax\ifhmode\unskip\space\fi MR }
\providecommand{\MRhref}[2]{%
  \href{http://www.ams.org/mathscinet-getitem?mr=#1}{#2}
}
\providecommand{\href}[2]{#2}
\begin{thebibliography}{Mum62}

\bibitem[AB83]{ab2}
MF~Atiyah and R.~Bott, \emph{{The Yang-Mills equations over Riemann surfaces}},
  Philos. Trans. Roy. Soc. London Ser. A \textbf{308} (1983), no.~1505,
  523--615.

\bibitem[BHH10]{biswas2009moduli}
I.~Biswas, J.~Huisman, and J.C. Hurtubise, \emph{{The moduli space of stable
  vector bundles over a real algebraic curve}}, Mathematische Annalen
  \textbf{347} (2010), no.~1, 201--233.

\bibitem[Dol63]{dold1963partitions}
A.~Dold, \emph{Partitions of unity in the theory of fibrations}, The Annals of
  Mathematics \textbf{78} (1963), no.~2, 223--255.

\bibitem[GH81]{gross1981real}
B.H. Gross and J.~Harris, \emph{Real algebraic curves}, Ann. Sci. {\'E}cole
  Norm. Sup.(4) \textbf{14} (1981), no.~2, 157--182.

\bibitem[GH04]{goldin2004real}
RF~Goldin and TS~Holm, \emph{Real loci of symplectic reductions}, Trans. Amer.
  Math. Soc \textbf{356} (2004), no.~11, 4623--4642.

\bibitem[Gro57]{grothendieck1957classification}
A.~Grothendieck, \emph{Sur la classification des fibr{\'e}s holomorphes sur la
  sphere de riemann}, American Journal of Mathematics \textbf{79} (1957),
  no.~1, 121--138.

\bibitem[HN75]{hn}
G.~Harder and MS~Narasimhan, \emph{{On the cohomology groups of moduli spaces
  of vector bundles on curves}}, Mathematische Annalen \textbf{212} (1975),
  no.~3, 215--248.

\bibitem[LS13]{liu2011yang}
C.C.M. Liu and F.~Schaffhauser, \emph{{Yang-Mills equations over Klein
  surfaces}}, arXiv:1109.5164v3 (2013).

\bibitem[McC01]{mccleary2001user}
J.~McCleary, \emph{{A user's guide to spectral sequences}}, Cambridge
  University Press, 2001.

\bibitem[Mil56a]{milnor1956construction}
J.~Milnor, \emph{{Construction of universal bundles, I}}, The Annals of
  Mathematics \textbf{63} (1956), no.~2, 272--284.

\bibitem[Mil56b]{milnor1956construction2}
\bysame, \emph{{Construction of universal bundles, II}}, The Annals of
  Mathematics \textbf{63} (1956), no.~3, 430--436.

\bibitem[MM65]{milnor1965structure}
J.W. Milnor and J.C. Moore, \emph{On the structure of {H}opf algebras}, The
  Annals of Mathematics \textbf{81} (1965), no.~2, 211--264.

\bibitem[MS74]{milnor1974cc}
J.W. Milnor and J.D. Stasheff, \emph{{Characteristic Classes}}, Princeton
  University Press, 1974.

\bibitem[Mum62]{mumford2004projective}
D.~Mumford, \emph{Projective invariants of projective structures and
  applications}, Proceedings of the ICM (1962), 23.

\bibitem[PS60]{ps}
R.~Palais and T.~Stewart, \emph{{Deformations of compact differentiable
  transformation groups}}, Amer. J. Math \textbf{82} (1960), 935--937.

\bibitem[Sch11]{schaffhauser2009moduli}
F.~Schaffhauser, \emph{{Moduli spaces of vector bundles over a Klein surface}},
  Geom. Dedicata \textbf{151} (2011), no.~1, 187--206.

\bibitem[Sch12]{schaffhauser2010real}
\bysame, \emph{Real points of coarse moduli schemes of vector bundles on a real
  algebraic curve}, Journal of Symplectic Geometry \textbf{10} (2012), no.~4,
  503--534.

\bibitem[Smi70]{smith1970lectures}
L.~Smith, \emph{Lectures on the {E}ilenberg-{M}oore spectral sequence},
  Springer-Verlag, 1970.

\bibitem[SW10]{saveliev2010real}
N.~Saveliev and S.~Wang, \emph{On real moduli spaces of holomorphic bundles
  over {M}-curves}, Topology and its Applications \textbf{158} (2010), no.~3,
  344--351.

\end{thebibliography}

\end{document}